\documentclass[a4paper, reqno]{amsart}
\allowdisplaybreaks

\title{Quantum Hamiltonian reductions for $\W$-algebras}

\author[Justine Fasquel]{Justine Fasquel}
\address[J.F.]{Université Bourgogne Europe, CNRS, IMB UMR 5584, 21000 Dijon, France}
\email{justine.fasquel@u-bourgogne.fr}

\author[Shigenori Nakatsuka]{Shigenori Nakatsuka}
\address[S.N.]{Department Mathematik, FAU Erlangen–Nürnberg, Cauerstraße 11, 91058, Erlangen, Germany}
\email{shigenori.nakatsuka@fau.de}

\usepackage{stmaryrd}
\usepackage{array}
\usepackage{latexsym}
\usepackage{amsmath}
\usepackage{amsfonts}
\usepackage{amssymb,nccmath}
\usepackage{mathtools}
\usepackage{amsxtra}
\usepackage{mathrsfs}
\usepackage{eucal}
\usepackage[all]{xy}
\usepackage{color}
\usepackage{braket}
\usepackage{graphics, setspace}
\usepackage{etoolbox, textcomp}
\usepackage{comment}
\usepackage{changepage}
\usepackage{xcolor}
\definecolor{rouge}{rgb}{0.85,0.1,.4}
\definecolor{bleu}{rgb}{0.1,0.2,0.9}
\definecolor{violet}{rgb}{0.7,0,0.8}
\usepackage[colorlinks=true,linkcolor=bleu,urlcolor=violet,citecolor=rouge]{hyperref}
\usepackage[capitalise]{cleveref}
\usepackage{lscape}
\usepackage{caption}
\usepackage{subcaption}
\usepackage{enumitem}
\usepackage{graphicx}
\usepackage{arydshln}
\usepackage{dynkin-diagrams}
\usepackage{bbm}
\usepackage{blindtext}

\usepackage[backend=biber,safeinputenc,sorting=nyt,maxbibnames=10]{biblatex}
\newcommand\doi[2]{\href{http://dx.doi.org/#1}{#2}}

\usepackage[aligntableaux=center]{ytableau}
\ytableausetup{mathmode,boxsize=0.4em}

\usepackage{tikz}			
\usetikzlibrary{calc} \tikzset{>=latex} \usetikzlibrary{backgrounds} \usetikzlibrary{shapes.geometric}
\usetikzlibrary{matrix}
\usetikzlibrary{automata, arrows.meta, positioning, cd}
\usepackage{dynkin-diagrams}

\newtheorem{definition}{Definition}[section]
\newtheorem{proposition}[definition]{Proposition}
\newtheorem{theorem}[definition]{Theorem}

\newtheorem{corollary}[definition]{Corollary}
\newtheorem{lemma}[definition]{Lemma}

\theoremstyle{remark}
\newtheorem{remark}[definition]{Remark}

\numberwithin{equation}{section}

\newcommand{\Z}{\mathbb{Z}}     
\newcommand{\Q}{\mathbb{Q}}     
\newcommand{\C}{\mathbb{C}}     

\newcommand{\D}{\mathcal{D}}
\newcommand{\W}{\mathcal{W}}    

\newcommand{\V}{\mathcal{V}}    

\newcommand{\HH}{\mathrm{H}}    
\newcommand{\OO}{\mathbb{O}}    

\newcommand{\tW}{\texorpdfstring{$\W$}{W}}

\newcommand{\Lie}{\mathrm{Lie}}

\newcommand{\BF}[1]{\mathbf{#1}}

\newcommand{\chfermion}{\mathcal{F}_\mathrm{ch}}

\newcommand{\End}{\operatorname{End}}
\newcommand{\Hom}{\operatorname{Hom}}
\renewcommand{\ker}{\operatorname{Ker}}
\newcommand{\im}{\operatorname{Im}}
\newcommand{\Span}{\operatorname{Span}}
\renewcommand{\End}{\mathrm{End}}

\newcommand{\kk}{\mathsf{k}}   
\newcommand{\ee}{\mathsf{e}}  
\newcommand{\wun}{\mathbbm{1}}   
\newcommand{\dd}{\mathrm{d}}
\newcommand{\hv}{\mathsf{h}^\vee}   

\newcommand{\g}{\mathfrak{g}}   
\newcommand{\h}{\mathfrak{h}}   
\newcommand{\n}{\mathfrak{n}}
\newcommand{\np}{\mathfrak{n}}
\newcommand{\Np}{N}

\newcommand{\sll}{\mathfrak{sl}}    

\newcommand{\E}{\mathrm{E}}

\newcommand{\SL}{\mathrm{SL}}       

\newcommand{\lbr}[2]{[{\hspace{1mm}}#1 {\hspace{1mm}} {}_\lambda {\hspace{1mm}} #2 {\hspace{1mm}}]}
\newcommand{\CT}{\mathrm{CT}}

\newcommand{\affWak}[1]{\mathbb{W}^{\kk}_{{#1}}}

\newcommand{\fockIO}[2][\alpha]{\mathrm{e}^{-\frac{1}{\kk+\hv}{#1}_{#2}}}
\newcommand{\dz}{\mathrm{d}z}
\newcommand{\oo}[1]{\OO_{[#1]}}
\newcommand{\ad}{\mathrm{ad}}

\newcommand{\heis}[2][\h]{\pi_{#1}^{#2}}
\newcommand{\Fock}[2]{\pi_{\h,#1}^{#2}}

\newcommand{\bg}[1]{\mathcal{A}_{#1}}

\newcommand{\no}[1]{\mathopen{:} #1 \mathclose{:}}  

\newcommand{\diag}[1]{\mathrm{diag}(#1)}
\renewcommand{\epsilon}{\varepsilon}

\newcommand{\KL}{\mathrm{KL}}
\newcommand{\weyl}[1]{\mathbb{V}_{#1}}

\renewcommand{\mod}{\operatorname{-Mod}}

\newcolumntype{M}[1]{>{\centering\arraybackslash}m{#1}}

\newcommand{\numtableaux}[1]{
\ytableausetup{boxsize=1.15em,aligntableaux=bottom}
\begin{ytableau}
    #1
\end{ytableau}
\ytableausetup{boxsize=0.4em}}
\newcommand{\ie}{\textit{ie.}}
\newcommand{\eg}{\textit{eg.}}

\begin{document}

\begin{abstract}
In this paper, we establish a general criterion for good pairs, namely pairs consisting of a nilpotent orbit and an even good grading in a simple Lie algebra, which guarantees the existence of a quantum Hamiltonian reduction between affine $\W$-algebras. In particular, we show that for type $A$, any two affine $\W$-algebras associated with adjacent nilpotent orbits are related by quantum Hamiltonian reductions.
\end{abstract}
\maketitle

\section{Introduction}
From a finite-dimensional simple Lie algebra $\g$ over the complex numbers $\C$, one can construct a vertex algebra $\V^\kk(\g)$, that depends on a complex parameter $\kk$. This vertex algebra is called the affine vertex algebra and plays an analogous role to the enveloping algebra of the Kac-Moody algebra $\widehat{\g}$. More precisely, its representation theory captures the smooth representations over the affine Lie algebra $\widehat{\g}$ at level $\kk$.

The affine vertex algebra $\V^\kk(\g)$ serves as the foundational object from which we can construct a large class of vertex algebras called the $\W$-algebras.
The $\W$-algebra $\W^\kk(\g,\OO)$ at level $\kk$ is defined \cite{FF90, KRW03} by applying to $\V^\kk(\g)$ a quantum Hamiltonian reduction corresponding to the nilpotent orbit $\OO\subset\g$:
\begin{align}\label{eq: def of Walg in intro}
    \W^\kk(\g,\OO)=\HH^0_{\OO}(\V^\kk(\g)).
\end{align}
This reduction uses the data of good grading $\Gamma$ for a representative nilpotent element $f\in\OO$. 
Nevertheless, it is known that the resulting $\W$-algebra depends only on the nilpotent orbit $\OO$.

$\W$-algebras are generally extensions of the Virasoro vertex algebra, \ie\ the symmetry algebra governing two-dimensional conformal symmetry in physics. As such, they play a central role in the study of conformal field theories and in many areas of mathematics connected to them. Originally arising in the context of integrable systems \cite{DSJKV, DS85, FF90}, $\W$-algebras play a fundamental role in the representation theory of affine Kac–Moody algebras \cite{Ad19, FRR, KW04}, the geometric Langlands program \cite{AF19, Fre07, Gai16}, and more recently the study of our-dimensional supersymmetric gauge theories via 4d/2d duality \cite{Ara18, BLL15, BPR15, SXY17} and the AGT correspondence \cite{BFN16, AGT10, SV13}.

Although each $\W$-algebra $\W^\kk(\g,\OO)$ is associated with a nilpotent orbit $\OO$ and is constructed independently as the quantum Hamiltonian reductions \eqref{eq: def of Walg in intro}, it has been noticed recently that they are indeed related in a way reflecting the closure relations of the nilpotent orbits with respect to the Zariski topology \cite{BN25, CFLN, FFFN, FKN24, GJ25}. 
More precisely, given two nilpotent orbits $\OO_1, \OO_2$ under the closure relation ${\OO}_1\subset \overline{\OO}_2$, the existence of a (partial or generalised) quantum Hamiltonian reduction $\HH_{\OO_2 \uparrow {\OO}_1}$, so that 
\begin{align}\label{statement of reduction by stages}
    \HH^n_{\OO_2 \uparrow {\OO}_1}\left(\W^\kk(\g,\OO_1) \right)\simeq \delta_{n,0} \W^\kk(\g,\OO_2)
\end{align}
holds, is expected.
Equation \eqref{statement of reduction by stages} can be regarded as a vertex-algebraic counterpart of similar known results for Whittaker models for the finite-dimensional Lie algebras (\ie\ finite $\W$-algebras) \cite{GJ23, Morgan14} or for the reductive algebraic groups over local fields \cite{GGS17}. 

The finite $\W$-algebras $\mathscr{U}(\g,\OO)$ \cite{Kostant78, Pre02} are associative algebras obtained from the enveloping algebras $\mathscr{U}(\g)$ by the quantum Hamiltonian reductions in parallel to the vertex-algebraic case {and thus can be viewed as a finite analogue of $\W^\kk(\g,\OO)$}. Indeed, $\mathscr{U}(\g,\OO)$ is also the Zhu algebra of $\W^\kk(\g,\OO)$ \cite{Ara07, DK06}.  
In this finite setting, partial reductions of the form \eqref{statement of reduction by stages} are better understood.
Indeed, Losev's decomposition theorem \cite{Losev10} asserts that the $\hbar$-adic version of $\mathscr{U}(\g)$, say $\mathscr{U}_\hbar(\g)$, is isomorphic to the tensor product of $\mathscr{U}_\hbar(\g,\OO)$ and a Weyl algebra $\mathcal{A}_{{\OO}\, \hbar}$ after appropriate completions;
hence a similar statement is naturally expected for $\mathscr{U}_\hbar(\g,\OO_1)$ with the tensor product of $\mathscr{U}_\hbar(\g,\OO_2)$ and some Weyl algebra, namely,
\begin{align}\label{statement of finite reduction by stages}
    \mathscr{U}_\hbar(\g,\OO_1)^{\wedge} \simeq \mathscr{U}_\hbar(\g,\OO_2)^{\wedge}\  \widehat{\otimes}_{\C[\![\hbar]\!]} \ \mathcal{A}_{\OO_2 \downarrow {\OO}_1\, \hbar}{}^{\wedge}.
\end{align}
Taking the quantum Hamiltonian reduction yields a finite analogue of \eqref{statement of reduction by stages}.

\subsection*{Main results}
In this paper, we consider the $\W$-algebras $\W^\kk(\g,\OO)$ obtained from representatives $f\in\OO$ which admit good even gradings (\ie\ $\Z$-gradings) $\Gamma$ on $\g$. The gradings are taken to be compatible with the standard triangular decomposition $\g=\n^+\oplus \h \oplus \n^-$. For simplicity, we denote $\np:=\n^+$ in the following.
Let $(f_i,\Gamma_i)$ ($i=1,2$) be two such pairs satisfying the following conditions:
\begin{itemize}
    \item Grading condition: $\np=\np_{0,0}\oplus \np_{0,1} \oplus \np_{1,0} \oplus \np^+_{+}$ with respect to the bi-grading $(\Gamma_1,\Gamma_2)$ where $\n^+_{+}=\bigoplus_{i,j>0}\n^+_{i,j}$;
    \item Nilpotent condition: $f_\circ:=f_2-f_1$ lies in $\n^{-}_{0,-1}$.
\end{itemize}
Then, one can construct a BRST complex $C_{\OO_2 \uparrow \OO_1}^\bullet(\W^\kk(\g,\OO_1))$ and take the associated quantum Hamiltonian reduction, \ie\ the cohomology  $\HH_{\OO_2 \uparrow \OO_1}^\bullet(\W^\kk(\g,\OO_1))$.
The main result is the following. 
\begin{theorem}\label{intro: Main thm: PQHR}
    For generic levels $\kk$, there is an isomorphism of vertex algebras 
    \begin{equation*}
        \HH^0_{\OO_2 \uparrow \OO_1}(\W^\kk(\g,\OO_1))\simeq \W^\kk(\g,\OO_2).
    \end{equation*}
\end{theorem}

Moreover, one can also assign a BRST complex $C_{\OO_2 \uparrow \OO_1}^\bullet(M)$ to any $\W^\kk(\g,\OO_1)$-modules $M$ in a similar way. Thus, we obtain a functor 
$$\HH^0_{\OO_2 \uparrow \OO_1}\colon \W^\kk(\g,\OO_1)\mod \rightarrow \W^\kk(\g,\OO_2)\mod.$$
The isomorphism in Theorem \ref{intro: Main thm: PQHR} is generalized to modules obtained from the Kazhdan--Lusztig category $\KL^\kk(\g)$ of the affine vertex algebra $\V^\kk(\g)$:
$$\HH_{\OO}^0\colon \KL^\kk(\g) \rightarrow \W^\kk(\g,\OO)\mod.$$

\begin{corollary}\label{intro: Main thm: PQHR for modules}
    For generic levels $\kk$ and $\V^\kk(\g)$-modules $M$ in $\KL^\kk(\g)$, there is an isomorphism of $\W^\kk(\g,\OO_2)$-modules $\HH^0_{\OO_2 \uparrow \OO_1}\circ\HH_{\OO_1}^0(M) \simeq \HH_{\OO_2}^0(M)$. Hence, the following functors are naturally isomorphic
    \begin{equation}\label{intro, eq: equivalence}
        \HH^0_{\OO_2}, \HH^0_{\OO_2 \uparrow \OO_1}\circ\HH_{\OO_1}^0\colon \KL^\kk(\g) \rightarrow \W^\kk(\g,\OO_2)\mod.
    \end{equation}
\end{corollary}

Finally, we show that, in type $A$, where we can always find even good gradings, the two conditions mentioned above are satisfied for arbitrary adjacent pairs of nilpotent orbits $\OO_1$ and $\OO_2$ in closure relation ${\OO}_1\varsubsetneq\overline{\OO}_2$ (denoted $\OO_1<\OO_2$). 
\begin{theorem}\label{intro: type A general adjacent}
    For generic levels $\kk$ and for arbitrary adjacent nilpotent orbits ${\OO_1<\OO_2}$ in $\sll_N$, there exists a quantum Hamiltonian reduction $\HH_{\OO_2\uparrow \OO_1}^0$ such that 
    \begin{align*}
        \W^\kk(\sll_N,\OO_2)\simeq \HH_{\OO_2\uparrow \OO_1}^0(\W^\kk(\sll_N,\OO_1)).
    \end{align*}
     Moreover, \eqref{intro, eq: equivalence} holds.
\end{theorem}

As a direct consequence, we have the existence of quantum Hamiltonian reductions $\HH_{\OO_2\uparrow \OO_1}^0$ between all the nilpotent orbits ${\OO_1<\OO_2}$ in closure relation in type $A$. The reduction functor $\HH_{\OO_2\uparrow \OO_1}^0$ is obtained by composing quantum Hamiltonian reductions appearing in Theorem \ref{intro: type A general adjacent} following a ``path'' of adjacent nilpotent orbits $\OO_1<\OO_{\alpha_1}<\dots<\OO_{\alpha_n}<\OO_2$.
\begin{corollary}
    For generic levels $\kk$ and for arbitrary nilpotent orbits ${\OO_1<\OO_2}$ in $\sll_N$, there exists a quantum Hamiltonian reduction $\HH_{\OO_2\uparrow \OO_1}^0$ such that 
    \begin{align*}
        \W^\kk(\sll_N,\OO_2)\simeq \HH_{\OO_2\uparrow \OO_1}^0(\W^\kk(\sll_N,\OO_1)).
    \end{align*}
    Moreover, we have the following commutative diagram:
    \begin{center}
    \begin{tikzcd}
        \KL^k(\g)\arrow{rr}{{\HH^0_{\OO_{2}} }} \arrow{rd}{{\HH^0_{\OO_{1}} }} && \W^k(\g,\OO_2)\mod \\
        & \W^k(\g,\OO_1)\mod. \arrow{ur}{{\HH^0_{\OO_2 \uparrow \OO_1}}} &
    \end{tikzcd}
    \end{center}
\end{corollary}

\subsection*{Organisation of the paper} The rest of the paper is organised as follows. 
In \S \ref{sec: Preliminary on W-algebras}, we introduce the set-up of
the $\W$-algebras. 
In \S \ref{sec: Wakimoto realization}, we recall Wakimoto realisations of $\W$-algebras, which are the main tools to establish the main results. 
In \S \ref{sec: Quantum Hamiltonian reductions}, we present the main results and Theorem \ref{intro: Main thm: PQHR} is proved in \S \ref{sec:proof_PQHR}. Finally, in \S \ref{sec: Case of type A}, we restrict to type $A$ and establish Theorem \ref{intro: type A general adjacent}.

\subsection*{Acknowledgements}
We thank Thibault Juillard for interesting discussions.
J.F.'s research was supported by a University of Melbourne Establishment Grant and an Andrew Sisson Support Package 2025. S.N. thanks the University of Melbourne for its hospitality during his stay while the result was obtained. 
Major progress was made while the authors stayed in Creswick for a Matrix Minor Program. We would like to thank all the Matrix Office members and staff for providing a welcoming environment and logistical support that made our work both productive and enjoyable.

\section{Preliminary on \tW-algebras}\label{sec: Preliminary on W-algebras}
\subsection{Simple Lie algebras}
Let $\g$ be a finite-dimensional simple Lie algebra over the complex numbers $\C$, $\g=\n^+\oplus \h \oplus \n^-$ a triangular decomposition, and $\g=\h \oplus \left(\bigoplus_{\alpha \in \Delta} \g_\alpha\right)$ the root space decomposition where $\Delta$ is the set of roots.
We decompose $\Delta=\Delta^{+}\sqcup \Delta^-$ into the sets of positive and negative roots, so that $\n^\pm=\bigoplus_{\alpha \in \Delta^\pm} \g_\alpha$, and denote by $\Pi=\{\alpha_1,\dots,\alpha_{\ell}\}\subset \Delta^+$ the set of simple roots.
We fix a set of corresponding (non-zero) root vectors $e_\alpha$ ($\alpha \in \Delta$) and coroots $\{h_1,\dots, h_\ell\}\subset \h$ so that $C=(\alpha_j(h_i))_{1\leq i,j\leq \ell}$ is the Cartan matrix of $\g$.
Then $\{e_\alpha \mid \alpha \in \Delta\}\cup\{h_i \mid i=1,\dots,\ell\}$ forms a basis of $\g$ and we denote by $c_{\lambda,\mu}^\nu$ the structure constants: $[h_i,e_\alpha]=c_{i,\alpha}^\alpha e_\alpha$ with $c_{i,\alpha}^\alpha=h_i(\alpha)$ and
$[e_\alpha,e_\beta]=\sum_{\gamma\in\Delta}c_{\alpha,\beta}^{\gamma}e_\gamma$.

\subsection{Nilpotent orbits and good gradings}\label{sec: Nilpotent orbits}
Let $\mathcal{N}\subset \g$ be the nilpotent cone, \ie\ the set consisting of elements $f\in\g$ such that $\ad(f)$ is nilpotent. It is closed under the adjoint action of the corresponding simple algebraic group $G$ of adjoint type whose Lie algebra is $\g$, which decomposes $\mathcal{N}$ into finitely many conjugacy classes, called the nilpotent orbits. 
In the case of  $\g=\sll_N$, the nilpotent orbits are parametrised by the partitions $\lambda=[\lambda_1^{m_1},\dots,\lambda_s^{m_s}]$ of $N$ --- with $\lambda_1>\dots >\lambda_s>0$ and the exponents $m_i$ denoting the multiplicities --- due to the Jordan classification. Setting $\OO_{\lambda}$ the nilpotent orbit corresponding to the partition $\lambda$,
we have 
$\mathcal{N}(\sll_N)=\bigsqcup_{\lambda\in \mathscr{P}_{N}} \OO_\lambda$ with $\mathscr{P}_{N}$ the set of partitions of $N$. 

Given a nilpotent element $f\in\g$, a $\frac{1}{2}\Z$-grading $\Gamma$ on $\g=\bigoplus_{d\in \frac{1}{2}\Z}\g_d$ is said \emph{good} for $f$ if the following conditions are satisfied:
\begin{align}
   f \in \g_{-1},\quad \ker \ad(f) \cap \g_+=0,\quad \g_{-}\subset \im \ad(f)
\end{align}
where we set $\g_{\pm}=\bigoplus_{d>0}\g_{\pm d}$.
The pair $(f,\Gamma)$ is then called a \emph{good pair} \cite{KRW03}. 
Two good pairs $(f,\Gamma)$, $(f',\Gamma')$ are equivalent if they are conjugated under the adjoint $G$-action.
Given a (non-zero) nilpotent element, the good gradings are classified up to equivalence \cite{EK05} and expressed in terms of pyramids \cite{BG05}. 

The pyramids are labelled by Young diagrams with shiftings.
We explain concretely with an example. 
The Young tableau of the nilpotent orbit $\oo{3,2}$ in $\g=\sll_5$ is the diagram  
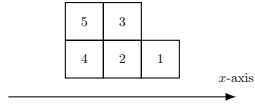
\begin{equation*}
    \begin{tikzpicture}[every node/.style={draw,regular polygon sides=4,minimum size=1cm,line width=0.04em},scale=0.5, transform shape]
        \node at (1,0)  {1};
        \node at (0,0)  {2};
        \node at (-1,0) {4};
        \node at (0,1)  {3};
        \node at (-1,1) {5};
        \draw[<-] (3,-1) --(-3,-1);
        \node[draw=none,regular polygon sides=0,minimum size=0pt,line width=0pt] at (3,-0.5) {{$x$-axis}};
    \end{tikzpicture}
\end{equation*}
Here, the boxes are labelled from right to left by the indices $\{1,2,3,4,5\}$ of the basis of the natural representation $\C^5$ of $\sll_5$.
From a pyramid, we read the nilpotent element $f$ given by
\begin{align}\label{association of nilpotent orbit}
    f=\sum_{i\rightarrow j}\E_{i,j}
\end{align}
where the summation is over indexes $(i,j)$ such that the boxes {\tiny $\numtableaux{i}$} and {\tiny $\numtableaux{j}$} are adjacent with {\tiny $\numtableaux{i}$} on the {left} of {\tiny $\numtableaux{j}$}; 
thus the Jordan class of $f$ agrees with the nilpotent orbit $\oo{3,2}$ we started with.

The grading is obtained by defining a $\frac{1}{2}\Z$-grading on the natural representation by reading the $x$-coordinates for the basis, which correspond to the eigenvalues of the adjoint action of the semisimple element
\begin{equation}\label{eq:def_grading}
    x=\sum_{i}\left(x_i\E_{i,i}\right)-\frac{1}{N}\left(\sum_{i}x_i\right)I_N
\end{equation}
where $x_i$ is the $x$-coordinate of the centre of the box {\tiny $\numtableaux{i}$}.
The good gradings defined in this way are compatible with the triangular decomposition in the sense that $\g_\pm \subset \n^\pm$ holds. This is the consequence of labelling the pyramids with the smaller numbers to the right. 
All the good pairs can be realised in this way up to conjugation by $G$.
We will use the decomposition of the root system $\Delta$ with respect to the good grading $\Gamma$:
\begin{align}
    \Delta_j=\{\alpha \in \Delta; \Gamma(\alpha)=j\},\quad \Delta_{>0}=\{\alpha \in \Delta; \Gamma(\alpha)>0\}. 
\end{align}

In this paper, we only consider \emph{even} good gradings, \ie\ those which are $\Z$-gradings. 
In particular, for all nilpotent orbits in type $A$, one can find a good even grading \cite{BG05,EK05}.
Therefore, a simple root $\alpha_i\in\Pi$ has grading $0$ or $1$. For convenience, we denote by $I_j$ ($j=0,1$) the sets of indexes $i=1,\dots,\ell$ such that $\Gamma(\alpha_i)=j$.

\subsection{\tW-algebras}\label{sec:W-algebras}
Given a simple Lie algebra $\g$, denote by
\begin{equation}
\widehat{\g}=\g[t^{\pm 1}] \oplus \ \C K
\end{equation}
the affine Kac--Moody algebra associated with $\g$. 
The universal affine vertex algebra $\V^\kk(\g)$ associated with $\g$ at level $\kk\in\C$ is the parabolic Verma $\widehat{\g}$-module
\begin{equation}
    \V^\kk(\g)=U(\widehat{\g})\otimes_{U(\g[t]\oplus\C K)}\C_\kk
\end{equation}
where $\C_\kk$ is the one-dimensional representation of $\g[t]\oplus\C K$ on which $\g[t]$ acts trivially and $K$ acts as the multiplication by the scalar $\kk$.
There is a unique vertex algebra structure on $\V^\kk(\g)$, which is strongly generated by the fields
\begin{equation}
    u(z)=Y(u_{-1},z)=\sum_{n\in\Z}u_n z^{-n-1},\qquad u_n=u\otimes t^n,\, u\in\g,
\end{equation}
satisfying the OPEs
\begin{align}
    u(z)v(w)\sim \frac{[u,v](w)}{(z-w)}+\frac{\kk(u,v)}{(z-w)^2}.
\end{align}
where $(\cdot, \cdot)$ is the normalised invariant bilinear form on $\g$.

In general, we denote by $Y(A,z)$ --- or briefly $A(z)$ ---
the field associated with an element $A$ in a vertex algebra, by $\wun$ the vacuum vector, and use the notation $AB$ (instead of $\no{AB}$) for the normally ordered product of $A$ and $B$ for simplicity. 

The $\W$-algebras are vertex algebras obtained as the quantum Hamiltonian reductions of the affine vertex algebras $\V^\kk(\g)$ \cite{FF90, KRW03}. 
Fix an \emph{even} good pair $(f, \Gamma)$ consisting of a (non-zero) nilpotent element $f$ and a $\Z$-grading $\Gamma$ good with respect to $f$.
Let $\bigwedge{}^{\bullet}_{\varphi,\varphi^*}$ denote the $bc$-system strongly and freely generated by the odd fields $\varphi(z),\varphi^*(z)$ satisfying the OPEs
\begin{align}
 \varphi(z)\varphi^*(w)\sim \frac{1}{(z-w)},\quad  \varphi(z)\varphi(w)\sim 0\sim \varphi^*(z)\varphi^*(w).
\end{align}
For each positively $\Gamma$-graded root $\alpha$ in $\Delta_{>0}$, let us denote by $\bigwedge{}^{\bullet}_{\varphi_\alpha,\varphi^*_\alpha}$ the copy of $\bigwedge{}^{\bullet}_{\varphi,\varphi^*}$ generated by the fields $\varphi_\alpha(z),\varphi^*_\alpha(z)$.

The BRST complex associated to the good pair $(f, \Gamma)$ is defined as
\begin{align}\label{BRST cohomology}
    C_f^\bullet(\V^\kk(\g))=\V^\kk(\g)\otimes \chfermion,\quad
    \chfermion:=\bigotimes_{\alpha \in \Delta_{>0}} \bigwedge{}^{\bullet}_{\varphi_\alpha,\varphi^*_\alpha}
\end{align}
and equipped with the differential 
\begin{align}\label{BRST differential}
    d=\int Y(Q,z)\ \dd z,\quad Q=Q_{\mathrm{st}}+Q_\chi
\end{align}
where
\begin{equation}
    Q_{\mathrm{st}}=\sum_{\alpha\in\Delta_{>0}} e_\alpha \varphi_\alpha^*-\frac{1}{2}\sum_{\alpha,\beta,\gamma\in\Delta_{>0}} c_{\alpha,\beta}^\gamma \varphi_\alpha^*\varphi_\beta^* \varphi_\gamma,\qquad
    Q_\chi=\sum_{\alpha\in\Delta_{>0}} (f,e_\alpha)\varphi_\alpha^*.
\end{equation}
The $\W$-algebra $\W^\kk(\g,f)$ associated with $\g$ and the good pair $(f,\Gamma)$ at level $\kk$ is defined as 
\begin{align}
    \W^\kk(\g,f)=\HH^0(C_f^\bullet(\V^\kk(\g)),d).
\end{align}
{Note that $\W^\kk(\g,f)$ is freely strongly generated by fields that correspond to a basis of the centraliser $\g^f$ homogeneous for the grading $\Gamma$ \cite{KW04}.}
However, $\W^\kk(\g,f)$ is independent as a vertex algebra of the choice of the representative $f$ in a given nilpotent orbit $\OO$ and the good grading $\Gamma$ \cite{AKM15, GJ25}, so we will occasionally use the notation $\W^\kk(\g,\OO)$ in the following.
In addition, replacing $\V^\kk(\g)$ by a $\V^\kk(\g)$-module $M$ in \eqref{BRST cohomology}, we obtain a functor from the category of $\V^\kk(\g)$-modules to the category of $\W^\kk(\g,\OO)$-modules, 
\begin{align}
    \HH^0_\OO\colon \V^\kk(\g)\mod \rightarrow \W^\kk(\g,\OO)\mod.
\end{align}

For later use, let us rewrite the BRST complex \eqref{BRST cohomology} using the semi-infinite cohomology for the Lie algebra $\g_+(\!(t)\!)$. Consider the one-dimensional $\g_+(\!(t)\!)$-module $\C_{\chi}$ on which $\g_+(\!(t)\!)$ acts by
\begin{equation}
    u_n \mapsto \mathrm{Res}_{t=0} (f,u)t^n \mathrm{d}t=\delta_{n,-1}(f,u).
\end{equation}
Consider the tensor product $\g_+(\!(t)\!)$-module $\V^\kk(\g)\otimes \C_\chi$  and define the semi-infinite complex
\begin{align}\label{semi-infinite complex}
    C^{\frac{\infty}{2}+\bullet}(\g_+(\!(t)\!); {\V^\kk(\g)}\otimes \C_\chi):= ({\V^\kk(\g)}\otimes \C_\chi)\otimes \chfermion
\end{align}
equipped with the differential $d_{\mathrm{st}}=\int Y(Q_{\mathrm{st}},z) \dz$, see \eqref{BRST differential}.
Then, \eqref{semi-infinite complex} agrees with the BRST complex $C_f^\bullet(\V^\kk(\g))$.
Replacing $\V^\kk(\g)$ by a $\V^\kk(\g)$-module $M$, we again obtain an isomorphism
\begin{align}\label{BRST is semi-infinite}
    \HH^\bullet_\OO(M)\simeq \HH^{\frac{\infty}{2}+\bullet}(\g_+(\!(t)\!);M\otimes \C_\chi).
\end{align}

\section{Wakimoto realisation}\label{sec: Wakimoto realization}
Let $G$ be a simple algebraic group (of adjoint type) whose Lie algebra is $\Lie(G)=\g$.
For simplicity, we denote the positive nilpotent part $\n^+:=\np$.
Let $\Np\subset G$ be the unipotent subgroup with $\Lie(\Np)=\np$.
Using the exponential map $\ee:\np\overset{\sim}{\longrightarrow}\Np$, one may identify the coordinate rings $\C[\Np]\simeq \C[z_\alpha\mid \alpha\in \Delta^+]$ where $z_\alpha$ is the coordinate of the root vector $e_\alpha\in\np$ ($\alpha \in \Delta^+$).
The left multiplication of $\Np$ on itself induces an anti-homomorphism $\rho^L:\np\to\D(\Np)$ where $\D(\Np)$ is the ring of differential operators on $\Np$. 
Then, one has 
\begin{equation}\label{eq:rho}
    \rho^L(e_{\alpha_i})=\sum_{\alpha\in \Delta^+}P_i^{\alpha}(z)\partial_{z_\alpha} 
\end{equation}
for some polynomials $P_i^{\alpha}(z)$ in $\C[\Np]$. 
More intuitively, by taking a faithful representation of $\Np$ and using the generic element $g(z)=g(z_\alpha|\alpha \in \Delta^+)$, the coefficients $P_i^{\alpha}(z)$ are uniquely determined by the equality 
\begin{equation}\label{eq:left_action}
    \ee^{\epsilon e_i} g(z)=g(z_\alpha+\epsilon P^\alpha_i(z)\mid \alpha \in \Delta^+)
\end{equation}
where $\epsilon$ is a dual number, \ie\ a variable satisfying $\epsilon^2=0$. 

The affine vertex algebra $\V^\kk(\g)$ admits a free fields realisation 
\begin{align}\label{Wakimoto realization for affine}
   \Psi\colon \V^\kk(\g)\hookrightarrow \bg{ \Delta^+ } \otimes \heis{\kk+\hv}
\end{align}
where $\heis{\kk+\hv}$ denotes the Heisenberg vertex algebra associated with the Cartan subalgebra $\h$ at level $\kk+\hv$ and $\bg{ \Delta^+ }$ is a tensor product of $|\Delta^+|$ copies of $\beta\gamma$-systems attached to each positive root.
This embedding is known as the \emph{Wakimoto realisation} (see \eg\ \cite{Fre07}). 
In addition, it gives rise to a family of $\V^\kk(\g)$-modules  $\affWak{\lambda}:=\bg{\Np}\otimes \Fock{\lambda}{\kk+\hv}$ ($\lambda\in \h^*$) obtained by restriction, called the \emph{Wakimoto representations}. Here $\Fock{\lambda}{\kk+\hv}$ is the Fock module over $\heis{\kk+\hv}$ of highest weight $\lambda$.
Note that \eqref{Wakimoto realization for affine} is restricted to 
\begin{align}
     \Psi^R_{\Np}\colon \V^0(\np)\hookrightarrow \bg{ \Delta^+ },
\end{align}
which is induced from the \emph{right} multiplication $\rho^R \colon \np\rightarrow \D(\Np)$. 

The affine vertex algebra $\V^\kk(\g)$ is isomorphic to the image of $\Psi$. At generic levels, it is equal to the intersection of the kernels of the screening operators
\begin{align}\label{affine screenings}
    S_i=\int S_i(z)\ \dd z \colon \affWak{0} \rightarrow \affWak{-\alpha_i},\quad S_i(z)= Y(\Psi^L_N(e_i)\fockIO{i},z),
\end{align}
associated with the simple root vectors $e_i$ ($1\leq i \leq \ell$). 
Here $\Psi^L_N(e_i)$ is the element in $\bg{\Delta^+}$ obtained from \eqref{eq:rho} by replacing $z_\alpha$ with $\gamma_\alpha$ and $\partial_{z_\alpha}$ with $\beta_\alpha$ (and taking normally ordered products when necessary).
Moreover, the embedding \eqref{Wakimoto realization for affine} is completed into an exact sequence at generic levels $\kk$ (such as $\kk \notin \Q$)
\begin{align}\label{Wakimoto resolution of affine}
    0\rightarrow  \V^\kk(\g) \overset{\Psi}{\longrightarrow} \affWak{0}\overset{\bigoplus S_i}{\longrightarrow} \bigoplus_{i=1,\ldots,\ell} \affWak{-\alpha_i}\rightarrow \mathrm{C}_2 \rightarrow \cdots \rightarrow \mathrm{C}_{n_\circ} \rightarrow 0
\end{align}
where
\begin{align}
    \mathrm{C}_i=\bigoplus_{\begin{subarray}c w\in W\\ \ell(w)=i\end{subarray}} \affWak{w\circ 0},
\end{align}
see for instance \cite[Proposition 4.2]{ACL19}. 
Here $W$ is the Weyl group of $\g$ and $\ell\colon W\rightarrow \Z_{>0}$ the length function with ${n_\circ}$ the length of the longest element, and $\circ\colon W\times \h^*\rightarrow \h^*$ the dot action given by $w\circ\lambda=w(\lambda+\rho)-\rho$ with $\rho$ the Weyl vector.

The Wakimoto realisations of the $\W$-algebras are obtained by applying $\HH_{\OO}^0$ to \eqref{Wakimoto resolution of affine}.
Since the $\V^\kk(\g)$-modules $\affWak{\lambda}$ are $\HH_{\OO}$-acyclic (see \eg\ \cite[Proposition 4.5]{Gen20}), we obtain the complex 
\begin{align}\label{Wakimoto resolution of W algebras}
    0\rightarrow  \W^\kk(\g,\OO) \rightarrow \HH_{\OO}^0(\affWak{0})\overset{\bigoplus S_i^\OO}{\longrightarrow} \bigoplus_{i=1,\ldots,\ell} \HH_{\OO}^0(\affWak{-\alpha_i})\rightarrow \cdots \rightarrow \HH_{\OO}^0(\mathrm{C}_{n_\circ}) \rightarrow 0,
\end{align}
which is exact by the cohomology vanishing $\HH_{\OO}^{\neq0}(\V^\kk(\g))=0$ \cite{KW04, Ara15a}.
The $\W^\kk(\g,\OO)$-modules $\HH_{\OO}^0(\affWak{\lambda})$ and the induced screening operators $S_i^\OO$ are explicitly obtained by using appropriate coordinates systems on $\Np$.

Setting $\np_0=\np\cap \g_0$, the subalgebra $\np$ decomposes as the semidirect product $\np= \np_0 \ltimes \g_+$. Correspondingly, we have the group semidirect product $\Np= \Np_0 \ltimes N^+$.
Then by \cite[Proposition 4.5, Theorem 4.8]{Gen20}, we have
\begin{align}\label{Wakimoto for Walg}
    \HH_{\OO}^0(\affWak{\lambda})\simeq \affWak{\OO,\lambda}:=\bg{\Delta^+_0} \otimes \Fock{\lambda}{\kk+\hv}
\end{align}
where $\Delta^+_0=\Delta^+\cap\Delta_0$, and the screening operators $S_i^\OO$ are expressed as
\begin{equation}\label{screeinings for Walg}
    S_i^\OO=\int S_i^\OO(z)\ \dd z,\quad S_i^\OO(z):=Y(P_i^\OO\fockIO{i},z)
\end{equation}
where 
\begin{align}
    P_i^\OO=\begin{cases}
         \displaystyle{\sum_{\alpha \in \Delta_{0}}P_i^{\alpha}}\beta_\alpha& (i \in I_0), \\
         \displaystyle{\sum_{\alpha \in \Delta^+}(f,e_\alpha)P_i^{\alpha}}& (i \in I_{1}).
    \end{cases}
\end{align}
Roughly, they are obtained from $S_i$ by evaluating $\beta_\alpha$ as $\beta_\alpha$ and  $(f,e_\alpha)$ depending on the grading $\Gamma(\alpha)=0,1$. 
Since, \eqref{eq:left_action} leads to
\begin{equation}
\ee^{\epsilon e_i}g=\ee^{\epsilon e_i}g_0g_+=g_0 \ \ee^{\epsilon (g_0 \ast e_i) }g_+
\end{equation}
where, we set $g \ast u= g^{-1}u g$ and use the decomposition $g=g_0g_+$ in $\Np=\Np_0\ltimes N^+$. 
Hence, we have more concisely
\begin{align}
    P_i^\OO=\Psi^L_{N_0}(e_i)\quad (i \in I_0),\qquad P_i^\OO=(f, g_0 \ast e_i)\quad (i \in I_1)
\end{align}

The Wakimoto realisation of the $\W$-algebras is as follows.
\begin{theorem}[\cite{Gen20}]\label{thm: Wakimoto general form}
At generic levels $\kk$, there is an isomorphism of vertex algebras 
\begin{align}\label{eq: Wakimoto for Walg}
    \W^\kk(\g,\OO)\simeq \bigcap_{i=1,\dots,\ell} \ker  S_i^\OO \subset \bg{\Delta_0^+}\otimes \heis{\kk+\hv}.
\end{align}
\end{theorem}
\begin{remark}
    The Wakimoto realisation still exists for non-generic levels, but the isomorphism in Theorem \ref{thm: Wakimoto general form} relaxed into an embedding.
\end{remark}

Recall that there is an embedding of vertex algebras
\begin{equation}
\V^{\kk^\natural}(\g_0^f) \hookrightarrow \W^\kk(\g,\OO)
\end{equation}
for some level $\kk^\natural$ \cite{KRW03}. The restriction $\V^{0}(\np_0^f)\subset \V^{\kk^\natural}(\g_0^f)$ can be realised explicitly through the Wakimoto realisation as follows.

\begin{proposition}\label{remaining Wakimoto}
For even good gradings $\Gamma$, the Wakimoto embedding 
$$\Psi^R_{\Np_0^f}\colon \V^0(\np_0^f)\hookrightarrow \bg{\Delta_0^+}$$ factors through the embedding \eqref{eq: Wakimoto for Walg}, \ie\ 
\begin{equation}
\Psi^R_{\Np_0^f}\colon \V^0(\np_0^f)\hookrightarrow  \W^\kk(\g,\OO)\subset \bg{\Delta_0^+} \otimes  \Fock{\lambda}{\kk+\hv}.
\end{equation}
\end{proposition}
\proof
By Theorem \ref{thm: Wakimoto general form}, it suffices to check 
$\lbr{P_i^\OO}{\rho^R_{\np_0^f}(u)}=0$ for $u\in \np_0^f$ and all $i$.
For $i \in I_0$, the commutativity of two actions (left and right multiplications $\Np_0 \curvearrowright \Np_0 \curvearrowleft \Np_0$) implies 
\begin{equation}
\lbr{P_i^\OO}{\rho^R_{\np_0^f}(u)}=\lbr{\rho^L_{\np_0^f}(e_i)}{\rho^R_{\np_0^f}(u)}=0.
\end{equation}
For $i \in I_1$, $P_i^\OO=(f, g_0 \ast e_i)$ is expressed only in terms of $\gamma_\alpha$ $(\alpha \in \Delta_0)$, and thus 
\begin{equation}
\lbr{P_i^\OO}{\rho^R_{\np_0^f}(u)}=-\rho^R_{\np_0^f}(u) (f, g_0 \ast e_i)
\end{equation}
where the right-hand side corresponds to the derivation of the function $(f, g_0 \ast e_i)$ by $\rho^R_{\np_0^f}(u)$. Hence, it is equal to the $\epsilon$-linear term, denoted by $\CT_\epsilon(?)$, of $(f, g_0 \ast e_i)$ with replacement $g_0\mapsto g_0\ee^{\epsilon u}$ by definition. Thus 
\begin{equation}
\begin{aligned}
    \lbr{P_i^\OO}{\rho^R_{\np_0^f}(u)}
    &=-\rho^R_{\np_0^f}(u) (f, g_0 \ast e_i)=\CT_\epsilon\left(-(f,(g_0\ee^{\epsilon u}) \ast e_i) \right)\\
    &=(f, [u,g_0 \ast e_i])=([f,u], g_0 \ast e_i)=0.
\end{aligned} 
\end{equation}
This completes the proof.
\endproof

\section{Quantum Hamiltonian reductions}\label{sec: Quantum Hamiltonian reductions}
Let us consider two good \emph{even} pairs $(f_i,\Gamma_i)$ ($i=1,2$) and set
\begin{equation}
    \g_{i,j}=\{x\in\g\mid\Gamma_1(x)=i\text{ and }\Gamma_2(x)=j\},\qquad \n^\pm_{i,j}=\g_{i,j}\cap\n^\pm
\end{equation}
where $i,j\in\Z$ that define
Assume that the bi-grading $(\Gamma_1,\Gamma_2)$
\begin{equation}
    \g=\bigoplus_{i,j\in \Z}\g_{i,j}\supset \n^\pm=\bigoplus_{i,j\in \Z}\n^\pm_{i,j}
\end{equation}
and the difference $f_o=f_2-f_1$ satisfy the following condition ($\bigstar$):
\begin{itemize}
    \item[($\bigstar$-1)] Grading condition: $\n(=\np^+)$ decomposes as 
        \begin{align}\label{decomp of nilp}
      \np=\np_{0,0}\oplus \np_{0,1} \oplus \np_{1,0} \oplus \np^+_{+},\qquad \np^+_{+}=\bigoplus_{i,j>0}\np_{i,j};
    \end{align}
    \item[($\bigstar$-2)] Nilpotent condition: $f_o$ lies in $\n^{-}_{0,-1}$.
\end{itemize}

Note in particular that the condition ($\bigstar$) implies the inclusion of the nilpotent orbits closures $\overline{\OO}_1\subset\overline{\OO}_2$ of $f_1,f_2$, that we denote by $\OO_1\leq\OO_2$.
The following properties are clear from the shape of the bi-grading. 
\begin{lemma}\label{consequences of the condition} 
Under the condition $(\bigstar)$ we have the following:
\begin{enumerate}
    \item The subalgebras $\np_{1,0}$ and $\np_{0,1}$ are abelian.
    \item The bilinear form
    \begin{align}\label{symplectic form}
        \omega: (\np_{0,1}/\np_{0,1}^{f_1}) \times \np_{1,0} \rightarrow \C,\quad (u,v)\mapsto (f_1,[u,v])
    \end{align}
    is non-degenerate.
\end{enumerate}
\end{lemma}
\proof
For (1), the assertion follows from 
\begin{equation}
    [\np_{1,0},\np_{1,0}]\subset \np_{2,0}=0,\qquad [\np_{0,1},\np_{0,1}]\subset \np_{0,2}=0.
\end{equation}
For (2), 
Since $f_1=f_2-f_\circ\in\n^-_{-1,-1}$, $[f_1,x]$ has value in $\n^-_{-1,0}$ for $x\in \np_{0,1}$. Moreover,
\begin{equation}
[f_1,\cdot\ ]\colon \np_{0,1}/\np_{0,1}^{f_1} \hookrightarrow \n^-_{-1,0}
\end{equation}
is injective.
Now, the assertion follows since $(\cdot,\cdot)\colon \n^-_{-1,0} \times \np_{1,0} \rightarrow \C$ is non-degenerate.
\endproof

\begin{remark}
    The condition presented in \cite{GJ25} for even gradings corresponds to the case $\np_{1,0}=0$.
\end{remark}

Now, thanks to Proposition \ref{remaining Wakimoto}, one may consider the quantum Hamiltonian reduction by using the following BRST complex:
\begin{align}
    C_{\OO_2 \uparrow \OO_1}^\bullet(\W^\kk(\g,\OO_1))=\W^\kk(\g,\OO_1)\otimes  \chfermion,\quad
    \chfermion:=\bigotimes_{\alpha \in \mathcal{B}_{0,1}^{f_1}} \bigwedge{}^{\bullet}_{\varphi_\alpha,\varphi^*_\alpha}
\end{align}
equipped with the differential 
\begin{align}\label{BRST differential_bis}
    d_{\OO_2 \uparrow \OO_1}=\int Y(Q_{\OO_2 \uparrow \OO_1},z)\ \dd z,\quad Q_{\OO_2 \uparrow \OO_1}=\sum_{\alpha \in \mathcal{B}_{0,1}^{f_1}} \left(\alpha+(f_\circ,\alpha)\right)\varphi_\alpha^*,
\end{align}
where $\mathcal{B}_{0,1}^{f_1}$ is a basis of $\np_{0,1}^{f_1}$.
We denote by $\HH_{\OO_2 \uparrow \OO_1}^\bullet(\W^\kk(\g,\OO_1))$ the cohomology.
Note that the shape of the complex and the differential are similar to the usual BRST cohomology defined in \S\ref{sec:W-algebras}, but expressions are simpler than \eqref{BRST cohomology} and \eqref{BRST differential} as $\np_{0,1}$ and thus $\np_{0,1}^{f_1}$ is abelian by Lemma \ref{consequences of the condition}. 
We note that such a type of differential typically appears for the BRST reduction associated with \emph{the minimal nilpotent orbits in type $A$}.
This is the generalisation of the Virasoro and Bershadsky-Polyakov type reductions studied in \cite{FFFN, FKN24}, corresponding to the cases $\g=\sll_2,\sll_3$, respectively.
The main result is the following. 
\begin{theorem}\label{Main thm: PQHR}
    For generic levels $\kk$, there is an isomorphism of vertex algebras 
    \begin{equation*}
        \HH^0_{\OO_2 \uparrow \OO_1}(\W^\kk(\g,\OO_1))\simeq \W^\kk(\g,\OO_2).
    \end{equation*}
\end{theorem}

Let $\KL^\kk(\g)$ denote the Kazhdan--Lusztig category of $\V^\kk(\g)$-modules, \ie\ the category of $\V^\kk(\g)$-modules which are bounded from below by the conformal grading with finite-dimensional graded spaces.
For generic levels $\kk$, $\KL^\kk(\g)$ is semisimple and the simple modules are the Weyl modules 
\begin{equation*}
    \weyl{\lambda}^\kk=U(\widehat{\g})\otimes_{U(\g[t]\oplus\C K)}L_\lambda,
\end{equation*}
which are induced from the simple highest weight $\g$-modules $L_\lambda$ with dominant integral highest weights $\lambda \in P_+$. The isomorphism in Theorem \ref{Main thm: PQHR} is generalised to modules obtained from $\KL^\kk(\g)$ as follows.
\begin{theorem}\label{Main thm: PQHR for modules}
    For $\V^\kk(\g)$-modules $M$ in $\KL^\kk(\g)$ at generic levels $\kk$, there is an isomorphism of $\W^\kk(\g,\OO_2)$-modules  
    $\HH^0_{\OO_2 \uparrow \OO_1}\circ\HH_{\OO_1}(M) \simeq \HH_{\OO_2}(M)$.
    Moreover, the following functors are naturally isomorphic
    \begin{equation*}
        \HH^0_{\OO_2}, \HH^0_{\OO_2 \uparrow \OO_1}\circ\HH_{\OO_1}^0\colon \KL^\kk(\g) \rightarrow \W^\kk(\g,\OO_2)\mod.
    \end{equation*}
\end{theorem}

\section{Proof of Theorem \ref{Main thm: PQHR} and \ref{Main thm: PQHR for modules}}\label{sec:proof_PQHR}
We decompose the sets of roots $\Delta^+$ and those of simple roots $I$ into
\begin{align*}
    \Delta^+=\Delta_{0,0}\sqcup \Delta_{0,1}\sqcup \Delta_{1,0}\sqcup \Delta_{+}^+,\qquad  I=I_{0,0}\sqcup I_{0,1}\sqcup I_{1,0}\sqcup I_{1,1}.
\end{align*}
By setting 
$$\Delta_0^{\OO_1}=\Delta_{0,0} \sqcup \Delta_{0,1},\qquad \Delta_0^{\OO_2}=\Delta_{0,0} \sqcup \Delta_{1,0},$$ we obtain the following realisation of $\W$-algebras $\W^\kk(\g, \OO_i)$ for $i=1,2$ from  Theorem \ref{thm: Wakimoto general form}.
\begin{proposition}\label{thm: Wakimoto special form}
At generic levels $\kk$, there is an isomorphism of vertex algebras 
\begin{align*}
    &\W^\kk(\g,\OO_i)\simeq \bigcap_{i\in I} \ker  S_i^{\OO_i} \subset \bg{\Delta_0^{\OO_i}}\otimes\heis{\kk+\hv}
\end{align*}
with $i=1,2$ and
\begin{align*}
    P_i^{\OO_1}=\begin{cases}
         \Psi^L_{N_0}(e_i),& \\
         \Psi^L_{N_{0,1}}(g_0 \ast e_i),& \\
         (f_1,(g_0g_{0,1}) \ast e_i),& \\
         (f_1,g_0 \ast e_i),& 
    \end{cases}\
    P_i^{\OO_2}=\begin{cases}
         \Psi^L_{N_0}(e_i),&  (i \in I_{0,0}) \\
         (f_2,(g_0g_{1,0}) \ast e_i),&  (i \in I_{0,1}) \\
         \Psi^L_{N_{1,0}}(g_0 \ast e_i),&  (i \in I_{1,0}) \\
         (f_2,g_0 \ast e_i),&  (i \in I_{1,1}).
    \end{cases}
\end{align*}
\end{proposition}
\begin{proof}
The coefficients $P_i^{\OO_1}$ corresponding to the first family of the screening operators are obtained from \eqref{eq:left_action} 
where we write the element $g$ in the upper unipotent group $\Np$ as 
\begin{align*}
g= g_{0}g_{0,1}g_{1,0}g_{+},\quad \text{with}\ \Np= (\Np_{0,0}\times \Np_{0,1})\times (\Np_{1,0} \times \Np^+_{+}),
\end{align*}
while the coefficients $P_i^{\OO_2}$ are computed using
\begin{align*}
g= g_{0}g_{1,0}g_{0,1}g_{+},\quad \text{with}\ \Np= (\Np_{0,0}\times \Np_{1,0})\times(\Np_{0,1} \times \Np^+_{+}),
\end{align*}
Both decompositions respect the semidirect product $\Np=\Np_0\rtimes \Np^+$ of \S\ref{sec: Wakimoto realization} with respect to the grading $\Gamma_1$ and $\Gamma_2$ respectively.

Additionally, note that since $\np_{0,1}$ and $\np_{1,0}$ are abelian by Lemma \ref{consequences of the condition} (1), the exponential map $\ee \colon \np \xrightarrow{\simeq} \Np$ truncates on these subalgebras:
$$\ee\colon \np_{0,1}\xrightarrow{\simeq}\Np_{0,1},\quad \np_{1,0}\xrightarrow{\simeq}\Np_{1,0},\quad u\mapsto \ee^{u}=I+u,$$
simplifying the expression of the coefficients $P_i^{\OO_1}$ and $P_i^{\OO_2}$.
\end{proof}

Let us derive a more concrete presentation by splitting 
$$\n_{0,1}=\n_{0,1}^\omega \oplus \n_{0,1}^{f_1} $$
so that the bilinear form $\omega \colon \n_{0,1}^\omega \times \n_{1,0}\rightarrow \C $ (see Lemma \ref{consequences of the condition}) is non-degenerate. 
Take bases $(u_1,\cdots,u_n)$, $(u_{n+1},\cdots u_m)$ of $\n_{0,1}^\omega, \n_{0,1}^{f_1}$ respectively, and a dual base $(v_1,\cdots, v_n)$ of $\n_{1,0}$ with respect to $\omega$: $\omega(u_i,v_j)=\delta_{i,j}$.
Note that by setting $u_i^*=[v_i,f_1]$ and $v_i^*=[f_1,u_i]$, we have $$(u_i^*,u_j)=\delta_{i,j},\quad (v_i^*,v_j)=\delta_{i,j}.$$
Now, let us express $g_{0,1},g_{1,0}$ by using these bases:
$$g_{i,i+1}=e^{A_{i}}=1+A_{i},\qquad A_{0}=\sum_{1\leq j \leq m} {z_j} u_j,\quad A_{1}=\sum_{1\leq j \leq n} {\widehat{z}_j} v_j.$$
\begin{lemma}\label{list of screening}
We have the following:
\begin{align*}
    P_i^{\OO_1}=\begin{cases}
         \Psi^L_{N_0}(e_i)&\\
         \displaystyle{\sum_{1\leq j \leq m}(u_j^*,g_{0} \ast e_i) \beta_j}&\\
         \displaystyle{-\sum_{1\leq j \leq n}(v_j^*,g_{0} \ast e_i) \gamma_j}&\\
         (f_1,g_{0} \ast e_i),&
    \end{cases}\
    P_i^{\OO_2}=\begin{cases}
         \Psi^L_{N_0}(e_i)&  (i \in I_{0,0}) \\
         \displaystyle{(f_o,g_{0}\ast e_i)+\sum_{1\leq j \leq n}(u_j^*,g_{0} \ast e_i) \widehat{\gamma}_j}&  (i \in I_{0,1}) \\
         \displaystyle{\sum_{1\leq j \leq n}(v_j^*,g_{0} \ast e_i) \widehat{\beta}_j}&  (i \in I_{1,0}) \\
         (f_1,g_{0} \ast e_i),&  (i \in I_{1,1}).
    \end{cases}
\end{align*}
\end{lemma}
\proof
The coefficient $P_i^{\OO_1}$, except for $i$ in $I_{1,0}$, is immediate from the expression in Proposition \ref{thm: Wakimoto special form}.
For $i \in I_{1,0}$, we have
\begin{align*}
    P_i^{\OO_1}&=(f_1, (g_{0} e^{A_0})\ast e_i) =
    \left(f_1,\sum_{j\geq0}\frac{(-1)^j}{j!}\ad_{A_0}^j(g_0\ast e_i)\right)\\
    &=(f_1,-[A_0,g_0 \ast e_i])=-([f_1,A_0],g_0 \ast e_i)=-\sum_{1\leq j \leq m}(v_j^*,g_0 \ast e_i) \gamma_j
\end{align*}
where we used that $\ad_{A_0}^j(g_0\ast e_i)\in\np_{1,j}$ and $f_1\in \np_{-1,-1}$.

For $i$ in $I_{0,0}$ and $I_{1,0}$, the computations are clear from Proposition \ref{thm: Wakimoto special form} again. 
The case $i$ in $I_{1,1}$ follows from the characters $(f_1,\cdot\ ), (f_2,\cdot\ )\colon \np^+_+ \rightarrow \C$ being equal on $\np_+^+$.
For $i \in I_{0,1}$, we have
\begin{align*}
    P_i^{\OO_2}&=(f_2, (g_0 e^{A_1})\ast e_i) =
    \left(f_2,\sum_{j\geq0}\frac{(-1)^j}{j!}\ad_{A_1}^j(g_0\ast e_i)\right)\\
    &=(f_o,g_0\ast e_i)+([A_1,f_1],g_0 \ast e_i) 
    =(f_o,g_0\ast e_i)+\sum_{1\leq j \leq n}(u_j^*,g_0 \ast e_i) \widehat{\gamma}_j.
\end{align*}
This completes the proof.
\endproof

\begin{proof}[Proof of Theorem \ref{Main thm: PQHR}]
We have
\begin{align*}
    \V(\n_{0,1}^{f_1})\hookrightarrow \W^\kk(\g,\OO_1)\hookrightarrow \bg{\Delta_{0}^{\OO_1}}\otimes\heis{\kk+\hv},\quad u_i\mapsto \beta_i,\quad  (n<i\leq m)
\end{align*}
by Proposition \ref{remaining Wakimoto}, and thus the differential $d_{\OO_2 \uparrow \OO_1}$ in \eqref{BRST differential_bis} is expressed by 
\begin{align*}
    d_{\OO_2 \uparrow \OO_1}=\int Y(Q_{\OO_2 \uparrow \OO_1},z)\dz,\quad Q_{\OO_2 \uparrow \OO_1}=\sum_{n<i\leq m}(\beta_i+(f_o,u_i))\varphi_i^*.
\end{align*}
Hence, the Poincaré lemma implies the cohomology vanishing for the Wakimoto representations 
$$\HH_{\OO_2 \uparrow \OO_1}^\bullet(\affWak{\OO_1,\lambda})=\HH_{\OO_2 \uparrow \OO_1}^\bullet(\bg{N_0\times N_{0,1}}) \otimes \Fock{\lambda}{\kk+\hv} \simeq \delta_{\bullet,0}\bg{N_0\times N_{0,1}^\omega}\otimes \Fock{\lambda}{\kk+\hv}$$
and thus the resolution \eqref{Wakimoto resolution of W algebras} implies the isomorphism of vertex algebras 
\begin{align*}
    \HH_{\OO_2 \uparrow \OO_1}^0(\W^\kk(\g,\OO_1))\simeq \bigcap_{i\in I} \ker  [S_i^{\OO_1}] \subset \bg{N_0\times N_{0,1}^\omega}\otimes\heis{\kk+\hv}
\end{align*}
characterized by the induced screening operators
\begin{equation*}
    [S_i^{\OO_1}]=\int [S_i^{\OO_1}](z)\ \dd z,\quad [S_i^{\OO_1}]=[P_i^{\OO_1}]\fockIO{i},
\end{equation*}
where $[P_i^{\OO_1}]$ denotes the cohomology classes of $P_i^{\OO_1}$.
Hence, it suffices to compare $[P_i^{\OO_1}]$ and $P_i^{\OO_2}$ for all $i \in I$. It is straightforward to check 
\begin{align}\label{induced screening}
[P_i^{\OO_1}]=\begin{cases}
         P_i^{\OO_1}&  (i \notin I_{0,1}) \\
         -(f_o,g_0\ast e_i)+\sum_{1\leq j \leq n}(u_j^*,g_0 \ast e_i)&  (i \in I_{0,1}).
    \end{cases}
\end{align}
We then apply the Fourier transform along $N_{0,1}^\omega$:
\begin{align}\label{eq: chiral Fourier transform}
  \bg{N_0\times N_{0,1}^\omega}\otimes\heis{\kk+\hv} \xrightarrow{\simeq} \bg{N_0\times N_{1,0}}\otimes\heis{\kk+\hv},\qquad \beta_i,\gamma_i \mapsto -\widehat{\gamma}_i,\widehat{\beta}_i,  
\end{align}
which induces an isomorphism 
\begin{align}\label{eq: identify the Wakimoto}
    \HH_{\OO_2 \uparrow \OO_1}^\bullet(\affWak{\OO_1,\lambda})\simeq \affWak{\OO_2,\lambda}
\end{align}
as modules over the free-field algebra corresponding to the case $\lambda=0$. Under this identification, we have
$$[P_i^{\OO_1}]=\pm P_i^{\OO_2}$$
by comparing \eqref{induced screening} and Lemma \ref{list of screening}.
Since the sign factors do not affect the kernels of the associated screening operators, we obtain the assertion.
\end{proof}
\begin{proof}[{Proof of Theorem \ref{Main thm: PQHR for modules}}]

By \cite{ACL19}, the Weyl modules $\weyl{\lambda}^\kk$ admit resolutions by Wakimoto modules, generalising the case $\lambda=0$ of \eqref{Wakimoto resolution of affine}:
\begin{align}\label{eq: Wakimoto resolution of Weyl modules}
    0\rightarrow  \weyl{\lambda}^\kk \rightarrow \affWak{\lambda}\overset{\bigoplus S_{i,\lambda}}{\longrightarrow} \bigoplus_{i=1,\ldots,\ell} \affWak{s_i\circ \lambda}\rightarrow \mathrm{C}_2^\lambda \rightarrow \cdots \rightarrow \mathrm{C}_{n_\circ}^\lambda \rightarrow 0
\end{align}
with 
\begin{align*}
    \mathrm{C}_i^\lambda=\bigoplus_{\begin{subarray}c w\in W\\ \ell(w)=i\end{subarray}} \affWak{w\circ \lambda}.
\end{align*}
Here, $s_i\in W$ is the $i$-th simple reflection and $S_{i,\lambda}$ are the screening operators given by the formula
\begin{align}\label{product of screening operators}
    S_{i,\lambda}[n]=\int_{\Upsilon} S_{i,\lambda}(z_1)\dots S_{i,\lambda}(z_{n})\ \dd z_1\dots \dd z_{n}\colon \affWak{\lambda}\rightarrow \affWak{\lambda-n \alpha_i}
\end{align}
with $n=\lambda(h_i)+1$ for some local system $\Upsilon$ on the configuration space $Y_{n}=\{(z_1,\dots,z_{n})\mid z_p\neq z_q\}$, see \cite{ACL19} for details.

Similarly to \eqref{Wakimoto resolution of W algebras}, we obtain resolutions of the $\W^\kk(\g,\OO)$-modules $\HH_{\OO}^0(\weyl{\lambda}^\kk)$ 
\begin{align}\label{eq: Wakimoto resolution of Weyl modules 2}
    0\rightarrow  \HH_{\OO}^0(\weyl{\lambda}^\kk) \rightarrow \affWak{\OO,\lambda}\overset{\bigoplus S_{i,\lambda}^\OO}{\longrightarrow} \bigoplus_{i=1,\ldots,\ell} \affWak{\OO,s_i\circ \lambda}\rightarrow \cdots \rightarrow \HH_{\OO}^0(\mathrm{C}_{n_\circ}^\lambda) \rightarrow 0,
\end{align}
for $\OO=\OO_1,\OO_2$ by applying the functor $\HH_{\OO}^0$ to \eqref{eq: Wakimoto resolution of Weyl modules}. Here, the induced screening operators $S^\OO_{i,\lambda}$ are given by 
\begin{align}
S^\OO_{i,\lambda}[n]=\int_{\Upsilon} S_{i,\lambda}^\OO(z_1)\dots S_{i,\lambda}^\OO(z_{n})\ \dd z_1\dots \dd z_{n}\colon \affWak{\OO, \lambda}\rightarrow \affWak{\OO, \lambda-n \alpha_i}
\end{align}
by \cite[Proposition 8.1]{FKN24}. 
We apply the functor $\HH_{\OO_2 \uparrow \OO_1}^0$ to \eqref{eq: Wakimoto resolution of Weyl modules 2} for $\OO=\OO_1$ and obtain 
\begin{align}\label{Wakimoto resolution of Weyl modules 3}
\begin{split}
    0\rightarrow  \HH_{\OO_2 \uparrow \OO_1}^0\left(\HH_{\OO_1}^0(\weyl{\lambda}^\kk)\right) \rightarrow \HH_{\OO_2 \uparrow \OO_1}^0(\affWak{\OO_1,\lambda})&\overset{\bigoplus [S_{i,\lambda}^{\OO_1}]}{\longrightarrow} \bigoplus_{i=1,\ldots,\ell} \HH_{\OO_2 \uparrow \OO_1}^0(\affWak{\OO_1,s_i\circ \lambda})\\
    &\rightarrow \cdots \rightarrow \HH_{\OO_2 \uparrow \OO_1}^0(\HH_{\OO_1}^0(\mathrm{C}_{n_\circ}^\lambda)) \rightarrow 0,
    \end{split}
\end{align}
which gives an isomorphism 
\begin{align*}
    \HH_{\OO_2 \uparrow \OO_1}^0\left(\HH_{\OO_1}^0(\weyl{\lambda}^\kk)\right)\simeq 
    \bigcap_{i=1,\dots,\ell} \ker  [S_{i,\lambda}^{\OO_1}] \subset \HH_{\OO_2 \uparrow \OO_1}^0(\affWak{\OO_1,\lambda})
\end{align*}
of modules over $\HH_{\OO_2 \uparrow \OO_1}^0\left(\HH_{\OO_1}^0(\V^\kk(\g))\right)\simeq\W^\kk(\g,\OO_2)$ by Theorem \ref{Main thm: PQHR}. 
By using the Fourier transform \eqref{eq: chiral Fourier transform} and the identification of modules \eqref{eq: identify the Wakimoto}, we obtain 
\begin{center}
\begin{tikzcd}
\HH_{\OO_2 \uparrow \OO_1}^0(\affWak{\OO_1,\lambda}) \arrow{rr}{[S_{i,\lambda}^{\OO_1}]} \arrow{d}{\simeq} && \HH_{\OO_2 \uparrow \OO_1}^0(\affWak{\OO_1,s_i\circ\lambda}) \arrow{d}{\simeq}\\
\affWak{\OO_2,\lambda} \arrow{rr}{[S_{i,\lambda}^{\OO_1}]} && \affWak{\OO_2,s_i\circ\lambda}
\end{tikzcd}
\end{center}
for some induced homomorphism in the bottom, which we denote by $[S_{i,\lambda}^{\OO_1}]$ by abuse of notation. 
Again, it follows from \cite[Proposition 8.1]{FKN24} that 
\begin{align*}
    [S_{i,\lambda}^{\OO_1}][n]=\int_{\Upsilon} [S_{i,\lambda}^{\OO_1}](z_1)\dots [S_{i,\lambda}^{\OO_1}](z_{n})\ \dd z_1\dots \dd z_{n}
\end{align*}
and thus $[S_{i,\lambda}^{\OO_1}]$ is identified with $S_{i,\lambda}^{\OO_2}$ by applying the Fourier transform \eqref{eq: chiral Fourier transform} as in the proof of Theorem \ref{Main thm: PQHR}.
Hence, we obtain isomorphisms
$$\HH^0_{\OO_2 \uparrow \OO_1}\circ\HH_{\OO_1}(\weyl{\lambda}^\kk) \simeq \HH_{\OO_2}(\weyl{\lambda}^\kk)$$
of $\W^\kk(\g,\OO_2)$-modules. Since $\KL^\kk(\g)$ is semisimple with simple objects $\weyl{\lambda}^\kk$, the above isomorphisms immediately give rise to a natural isomorphism of the functors $\HH^0_{\OO_2 \uparrow \OO_1}\circ\HH_{\OO_1}^0$ and $\HH^0_{\OO_2}$. This completes the proof.
\end{proof}

\section{Adjacent orbits in type \texorpdfstring{$A$}{A}} \label{sec: Case of type 
A}
Recall that in type $A$, \ie\ $\g=\sll_N$, nilpotent orbits are parametrised by partitions $\mathscr{P}_N$ of $N$ and the partial order on the set of nilpotent orbits is given by the dominance order on $\mathscr{P}_{N}$:
let $\lambda=[\lambda_1, \lambda_2, \dots, \lambda_n]$ and $\mu=[\mu_1, \mu_2,\dots, \mu_n]$ be the partitions of $N$ which correspond to the orbits $\OO_1$, $\OO_2$, respectively,
\begin{equation}\label{eq:condition}
\OO_1 \leq \OO_2 \ \Longleftrightarrow\ \sum_{1\leq j \leq i} \lambda_j \leq \sum_{1\leq j \leq i} \mu_j\quad (1\leq i\leq n).     
\end{equation}
Here, we have $\lambda_i\geq \lambda_{i+1}$ and $\mu_i\geq \mu_{i+1}$, but $\lambda_i, \mu_i$ might be equal to zero for convenience.
Moreover, two nilpotent orbits $\OO_1<\OO_2$ are called \emph{adjacent} if there is no intermediate orbit $\OO_1< \OO_3 <\OO_2$.

In this section, we show that any pair of adjacent nilpotent orbits in $\sll_N$ verifies the condition ($\bigstar$) of \S\ref{sec: Quantum Hamiltonian reductions}, and thus we obtain the following result by Theorem \ref{Main thm: PQHR}.

\begin{theorem}\label{type A general adjacent}
    Let $\kk$ be a generic level.
    For two adjacent nilpotent orbits $\OO_1<\OO_2$ in $\g=\sll_N$, there exists a quantum Hamiltonian reduction of the minimal type $\HH_{\OO_2\uparrow \OO_1}^0$ such that 
    \begin{align*}
        \W^\kk(\sll_N,\OO_2)\simeq \HH_{\OO_2\uparrow \OO_1}^0(\W^\kk(\sll_N,\OO_1)).
    \end{align*}
\end{theorem}

The rest of this section will be devoted to proving Theorem \ref{type A general adjacent} by specifying good pairs satisfying ($\bigstar$). We start with recalling the following well-known result; see for instance \cite{Bry}.

\begin{lemma}\label{lem: adjacent nilp}
    Two nilpotent orbits $\OO_1 , \OO_2$ of $\g=\sll_N$ with $\OO_1 <\OO_2$ are adjacent if and only of there exist $1\leq i < j \leq n$ such that 
    \begin{gather}
        \lambda_i=\mu_i-1,\quad {\lambda_{i+1}=\lambda_{i+2}=\dots=\lambda_{j}=\mu_j+1},\quad\lambda_k=\mu_k\quad (k \neq i,j)\label{cond:adj1}\\
        {\text{and either}\quad j=i+1\quad\text{or}\quad \lambda_i=\lambda_{i+1}.}\label{cond:adj2}
    \end{gather}
\end{lemma}
\begin{proof}
We include a proof for reader's convenience.
    We have $\OO_1 <\OO_2$ since \eqref{eq:condition} is satisfied. Assume $\OO_1\leq\OO_3<\OO_2$ with $\OO_3$ corresponding to the partition $\nu=[\nu_1,\nu_2,\dots,\nu_n]$.
    The dominance order implies directly that $\nu_k=\lambda_k=\mu_k$, for $k<i$. 
    Moreover, since $\nu\neq\mu$ and
    $$\begin{gathered}
        \left(\sum_{1\leq\ell\leq k}\mu_\ell\right)-1\leq \sum_{1\leq\ell\leq k}\nu_\ell\leq \sum_{1\leq\ell\leq k}\mu_\ell,\qquad (i\leq k<j)\\
        \text{and}\quad\sum_{1\leq\ell\leq k}\nu_\ell= \sum_{1\leq\ell\leq k}\mu_\ell,\qquad (k\geq j).\\
    \end{gathered}$$
    there is exactly one of the $i\leq k_\circ\!<j$ such that $\nu_{k_\circ}\!=\mu_{k_\circ}\!-1$, $\nu_k=\mu_k$ for $k\neq k_\circ$, and $\nu_k=\lambda_k$ for $k\geq j$. 
 
    If $j=i+1$, we have $k_\circ\!=i$ and thus $\nu=\lambda$.
    Otherwise, 
    assume $k_\circ\!\neq i$, then $\mu_{k_\circ}\!=\lambda_{k_\circ}$ and 
    $\nu_{k_\circ}\!=\lambda_{k_\circ}\!-1=\lambda_j-1=\nu_j-1<\nu_j$
    leads to a contradiction. Therefore, $k_\circ\!=i$ again and $\nu=\lambda$.
\end{proof}

The adjacency of nilpotent orbits can be illustrated pictorially on the Young tableaux by moving a box from a row down to a lower row such that all rows in between have the same size. For instance, the next figure illustrates three adjacent nilpotent orbits of $\sll_{14}$ corresponding to, from left to right, partitions $\lambda=[5,3,3,3]$, $\mu=[5,4,3,2]$, $\nu=[6,3,3,2]$ --- the box in gray is the one we drop at each step:
\begin{equation*}
\begin{tikzpicture}[every node/.style={draw,regular polygon sides=4,minimum size=1cm,line width=0.04em},scale=0.5, transform shape]
        \node[fill=gray] at (1,3)  {};
        \node at (1,2)  {};
        \node at (1,1)  {};
        \node at (1,0)  {};
        \node at (0,3)  {};
        \node at (0,2)  {};
        \node at (0,1)  {};
        \node at (0,0)  {};
        \node at (-1,3) {};
        \node at (-1,2) {};
        \node at (-1,1) {};
        \node at (-1,0) {};
        \node at (2,0) {};
        \node at (3,0) {};
    \end{tikzpicture},\hspace{2cm}
    \begin{tikzpicture}[every node/.style={draw,regular polygon sides=4,minimum size=1cm,line width=0.04em},scale=0.5, transform shape]
        \node at (-1,3)  {};
        \node at (1,2)  {};
        \node at (1,1)  {};
        \node at (1,0)  {};
        \node at (0,3)  {};
        \node at (0,2)  {};
        \node at (0,1)  {};
        \node at (0,0)  {};
        \node at (-1,2) {};
        \node at (-1,1) {};
        \node at (-1,0) {};
        \node at (2,0) {};
        \node[fill=gray] at (2,1) {};
        \node at (3,0) {};
    \end{tikzpicture},\hspace{2cm}
    \begin{tikzpicture}[every node/.style={draw,regular polygon sides=4,minimum size=1cm,line width=0.04em},scale=0.5, transform shape]
        \node at (-1,3)  {};
        \node at (1,2)  {};
        \node at (1,1)  {};
        \node at (1,0)  {};
        \node at (0,3)  {};
        \node at (0,2)  {};
        \node at (0,1)  {};
        \node at (0,0)  {};
        \node at (-1,2) {};
        \node at (-1,1) {};
        \node at (-1,0) {};
        \node at (2,0) {};
        \node at (3,0) {};
        \node[fill=gray] at (4,0) {};
    \end{tikzpicture}.
\end{equation*}

\begin{remark}
{The proof of Theorem \ref{type A general adjacent} only uses the condition \eqref{cond:adj1} of Lemma \ref{lem: adjacent nilp}. Hence, we also obtain the quantum Hamiltonian reductions between non-adjacent nilpotent orbits satisfying \eqref{cond:adj1}, such that $\lambda=[5,3,3,3]$ and $\nu=[6,3,3,2]$ in the above example.}
\end{remark}

Before specifying the representative of $\OO_1=\OO_\lambda$ and the associated good grading, we prepare some useful lemmas which hold for general even good pairs $(f,\Gamma)$ obtained from a Young tableau as in \S\ref{sec: Nilpotent orbits}.

\begin{lemma}\label{lem: description of upper nilp} Let $(f,\Gamma)$ be a good even grading obtained from a Young tableau and written $\np=\bigoplus_{d\in\Z_{\geq0}} \np_d$ the $\Gamma$-grading decomposition.
Then, we have the following 
    \begin{enumerate}
        \item $\np_{d}=\Span \{\E_{i,j}\mid i<j,\ x(i)=x(j)+d\}$,
        \item $\np^f=\np_0^f.$
    \end{enumerate}
\end{lemma}
\begin{proof}
    The first assertion follows directly from the definition of $\Gamma$ (see \ref{eq:def_grading}) while the second is a consequence of the representation theory of $\sll_2$.
\end{proof}

Let $\sll_2=\{e,h,f\}$ denote the $\sll_2$-triple so that 
\begin{align}\label{eq: constracting the sl2-triple}
    e=\sum_{i \rightarrow j} c_{j,i}\E_{j,i},\quad h {=2x}=\sum_{i \rightarrow j} c_{j,i}(\E_{j,j}-\E_{i,i}),\quad f=\sum_{i \rightarrow j} \E_{i,j}
\end{align}
where $c_{j,i} \in \C^\times$ are some constants, see \eqref{association of nilpotent orbit} for the summation rule. 
Thanks to the representation theory of $\sll_2$, we have
$$\np_0^f = \np_0 \cap \{u\in \sll_N\mid \sll_2 \cdot u=0 \}. $$
By expressing the partition $\lambda$ as $\lambda=[\underline{\lambda}_1^{m_1},\dots,\underline{\lambda}_s^{m_s}]$ with $\underline{\lambda}_1>\dots > \underline{\lambda}_s>0$, the $\sll_2$-representation  $\C^N$, that corresponds to the embedding $\sll_2\hookrightarrow \sll_N $, decomposes as 
$$\C^N \simeq \bigoplus_{1\leq a \leq s}(\C^{\underline{\lambda}_a})^{m_a}$$
as $\sll_2$-modules. 
Consider the decomposition 
$$\End_{U(\sll_2)}{(\C^N)}=\bigoplus_{1\leq a,b\leq s} \Hom_\C(\C^{m_a}\otimes \C^{\underline{\lambda}_a},\C^{m_b}\otimes \C^{\underline{\lambda}_b}),$$
then Schur's lemma implies the following:

\begin{lemma}\label{lem: description of f-inv} 
The decomposition $\End_{U(\sll_2)}{(\C^N)}=\bigoplus_{1\leq a\leq s} \End_\C(\C^{m_a})$ induces
$$\np^f= \np^f_0= \left(\np_0\cap\End_{U(\sll_2)}(\C^N)\right)^f=\bigoplus_{1\leq a \leq s} (\np_0 \cap \End (\C^{m_a})).  $$
\end{lemma}

Each summand $\np_0 \cap \End (\C^{m_a})$ corresponds to the rectangular partition $\lambda_a=[\underline{\lambda}_a^{m_a}]$, and can be written more explicitly as follows:
\begin{equation*}
\np_0 \cap \End (\C^{m_a})=\Span\left\{\widehat{\E}_{i,j}\mid 1\leq i<j\leq m_a \right\},
\qquad \widehat{\E}_{i,j}=\sum_{0\leq k<n} \E_{i+m_ak,j+m_ak}
\end{equation*}
by considering the following Young tableau of $\lambda_a=[\underline{\lambda}_a^{m_a}]$
\begin{equation*}
\begin{tikzpicture}[every node/.style={draw,regular polygon sides=4,minimum size=1cm,line width=0.04em},scale=0.58, transform shape]
        \node at (2,3)  {$m_a$};
        \node at (2,2)  {$\vdots$};
        \node at (2,1)  {$2$};
        \node at (2,0)  {$1$};
        \node at (1,3)  {$2m_a$};
        \node at (1,2)  {$\vdots$};
        \node at (1,1)  {\scriptsize{$m_a\!\!+\!2$}};
        \node at (1,0)  {\scriptsize{$m_a\!\!+\!1$}};
        \node at (0,3)  {$\cdots$};
        \node at (0,2)  {};
        \node at (0,1)  {$\cdots$};
        \node at (0,0)  {$\cdots$};
        \node at (-1,3)  {$N_a$};
        \node at (-1,2)  {$\vdots$};
        \node at (-1,1)  {\scriptsize{$N_a'\!+\!2$}};
        \node at (-1,0)  {\scriptsize{$N_a'\!+\!1$}};
    \end{tikzpicture}
\end{equation*}
where $N_a'=m_a(\underline{\lambda}_a-1)$ and $N_a=m_a\underline{\lambda}_a$.

\begin{proof}[Proof for Theorem \ref{type A general adjacent}]
Thanks to Lemma \ref{lem: adjacent nilp}, there exist $1\leq i < j \leq n$ such that 
\begin{align}\label{eq: useful condition}
    \lambda_i=\mu_i-1,\quad {\lambda_{i+1}=\dots=\lambda_{j}=\mu_j+1},\quad \lambda_k=\mu_k\quad (k \neq i,j).
\end{align}    
We will divide the proof into two steps. We first show the case $(i,j)=(1,n)$ (Case I) and then the general case (Case II). 

\paragraph{\bf{Case I}} 
Assume first $i=1$ and $j=n$ ($n\geq2$). Then, we express $\lambda$ and $\mu$ as follows:
$$\lambda=[a,\underbrace{b,\dots,b}_{s},b],\qquad \mu=[a+1,\underbrace{b,\dots,b}_{s},b-1].$$
{Note that in the adjacent case, $a=b$ when $s\neq0$ but we shall not use this extra condition here.}
We take the good gradings associated with the following Young tableaux.
First, for the partition $\lambda$, we associate the left-aligned Young tableau so that the left-end box of each row has the same $x$-coordinate and put a number for each box in the unique way so that {\tiny $\numtableaux{i}$} sits in the right or below of {\tiny $\numtableaux{j}$} whenever $i<j$ holds.
For the partition $\mu$, we associate the Young tableau obtained from the one for $\lambda$ by shifting the boxes on the top by one on the left and dropping the left-end box to the bottom.
For example, for the adjacent partitions $\lambda=[3,3,3]$ and $\mu=[4,3,2]$ we have
\begin{equation}\label{eq: sliding the boxes adj}
\begin{tikzpicture}[every node/.style={draw,regular polygon sides=4,minimum size=1cm,line width=0.04em},scale=0.58, transform shape]
        \node at (3,0)  {1};
        \node at (2,0)  {4};
        \node[fill=gray!50] at (3,2)  {3};
        \node at (3,1)  {2};
        \node at (1,0)  {7};
        \node[fill=gray!50] at (2,2)  {6};
        \node at (2,1)  {5};
        \node[fill=gray!50] at (1,2) {9};
        \node at (1,1) {8};
    \end{tikzpicture},\hspace{2cm}
    \begin{tikzpicture}[every node/.style={draw,regular polygon sides=4,minimum size=1cm,line width=0.04em},scale=0.58, transform shape]
        \node at (3,0)  {1};
        \node at (2,0)  {4};
        \node[fill=gray!50] at (0,0)  {9};
        \node at (3,1)  {2};
        \node at (1,0)  {7};
        \node[fill=gray!50] at (2,2)  {3};
        \node at (2,1)  {5};
        \node[fill=gray!50] at (1,2) {6};
        \node at (1,1) {8};
    \end{tikzpicture}
\end{equation}
but we can also consider pairs of non-adjacent partitions such as $\lambda=[5,3,3,3]$ and $\mu=[6,3,3,2]$ for which we fix the Young tableaux.
\begin{equation}\label{eq: sliding the boxes}
\begin{tikzpicture}[every node/.style={draw,regular polygon sides=4,minimum size=1cm,line width=0.04em},scale=0.58, transform shape]
        \node at (3,0)  {3};
        \node at (2,0)  {7};
        \node[fill=gray!50] at (3,3)  {6};
        \node at (3,2)  {5};
        \node at (3,1)  {4};
        \node at (1,0)  {11};
        \node[fill=gray!50] at (2,3)  {10};
        \node at (2,2)  {9};
        \node at (2,1)  {8};
        \node at (4,0)  {2};
        \node[fill=gray!50] at (1,3) {14};
        \node at (1,2) {13};
        \node at (1,1) {12};
        \node at (5,0) {1};
    \end{tikzpicture},\hspace{2cm}
    \begin{tikzpicture}[every node/.style={draw,regular polygon sides=4,minimum size=1cm,line width=0.04em},scale=0.58, transform shape]
        \node at (3,0)  {3};
        \node at (2,0)  {7};
        \node[fill=gray!50] at (2,3)  {6};
        \node at (3,2)  {5};
        \node at (3,1)  {4};
        \node at (1,0)  {11};
        \node[fill=gray!50] at (1,3)  {10};
        \node at (2,2)  {9};
        \node at (2,1)  {8};
        \node at (4,0)  {2};
        \node[fill=gray!50] at (0,0) {14};
        \node at (1,2) {13};
        \node at (1,1) {12};
        \node at (5,0) {1};
    \end{tikzpicture}.
\end{equation}
In the above examples, the coloured boxes correspond to those shifted in the process of obtaining $\mu$ from $\lambda$. 

Next, we take the standard representative for the nilpotent element for $\lambda$ 
$$f_\lambda=\sum_{i\underset{\lambda}{\rightarrow} j}\E_{i,j}$$
as defined in \eqref{association of nilpotent orbit}.
Here, we use the notation $i\underset{\lambda}{\rightarrow} j$ to emphasize that the summation is done with respect to the Young tableau associated with the partition $\lambda$.
On the other hand, we take a non-standard representative $\widetilde{f}_\mu\in\OO_\mu$ obtained as a conjugation of the standard one   
\begin{align}\label{eq: non-standard nilp}
    \widetilde{f}_\mu= g\ast {f}_\mu,\qquad {f}_\mu=\sum_{i\underset{\mu}{\rightarrow} j}\E_{i,j}
\end{align}
where $g \in \SL_N$ is an element with $\Gamma_\mu$-grading zero which we will specify later.  

By Lemma \ref{lem: description of upper nilp}, we have
\begin{align*}
    \n_{0,1}
   &=\Span \{\E_{i,j}\mid i<j,\  x_\lambda(i)=x_\lambda(j),\ x_\mu(i)=x_\mu(j)+1 \}\\
      &=\Span \{\E_{\BF{r}+i+j\BF{s},\BF{r}+(j+1)\BF{s}}\mid  0<i< \BF{s},\ 0\leq j <b-1\}
\end{align*}
where $\BF{r}=a-b$ and $\BF{s}=s+2$. Then it follows from Lemma \ref{lem: description of f-inv} that 
\begin{align*}
    \n_{0,1}^{f_\lambda} =\Span \left\{\widehat{\E}_{i}\mid  0<i< \BF{s} \right\},\quad \widehat{\E}_{i}=\sum_{j=0}^{b-1} \E_{\BF{r}+i+j\BF{s},\BF{r}+(j+1)\BF{s}}
\end{align*}
which is indeed an abelian subalgebra. The character used in the quantum Hamiltonian reduction in \eqref{BRST differential_bis} is then given by
$$(f_\circ,\ \cdot\ )\colon \n_{0,1}^{f_\lambda} \rightarrow \C,\quad\text{with}\quad f_\circ= (\widehat{\E}_{1})^{\mathrm{op}}= \sum_{j=0}^{b-1} \E_{\BF{r}+(j+1)\BF{s},\BF{r}+1+j\BF{s}}.$$

We show that $\widetilde{f}_\mu:=f_\lambda+f_\circ$ lies in the orbit $\OO_2=\OO_\mu$ by constructing $g$ in \eqref{eq: non-standard nilp}.
Let decompose $\C^N$ and $f_\mu$ row-wise, \ie
$$\C^N=\C^{a+1} \oplus \underbrace{\C^b \oplus \dots \oplus \C^b}_{s} \oplus \C^{b-1},\quad f_\mu=\BF{f}_\mu^0 + (\BF{f}_\mu^1+\dots +\BF{f}_\mu^s)+\BF{f}_\mu^{s+1}, $$
{where $\BF{f}_\mu^k$ is given by the sum in \eqref{eq: non-standard nilp} restricted to boxes {\tiny $\numtableaux{i}$} and {\tiny $\numtableaux{j}$} on the $(k+1)$-th row.}
We have a similar decomposition for $f_\lambda=\BF{f}_\lambda^0 + (\BF{f}_\lambda^1+\cdots +\BF{f}_\lambda^s)+\BF{f}_\lambda^{s+1}$ and note that
$$\BF{f}_\mu^1+\cdots +\BF{f}_\mu^s= \BF{f}_\lambda^1+\cdots +\BF{f}_\lambda^s.$$
It suffices to construct $g$ inside $\SL(\C^{a+1}\oplus \C^{b-1})\subset \SL(\C^N)=\SL_N$ such that 
$$
g^{-1}\ast\left(\BF{f}_\lambda^0 +\BF{f}_\lambda^{s+1}+f_\circ\right)=\BF{f}_\mu^0+\BF{f}_\mu^{s+1}
$$
since
\begin{align*}
g^{-1}\ast\widetilde{f}_\mu
&=g^{-1}\ast\left(\BF{f}_\lambda^0 + (\BF{f}_\lambda^1+\cdots +\BF{f}_\lambda^s)+\BF{f}_\lambda^{s+1}+f_\circ\right) \\
&=g^{-1}\ast\left(\BF{f}_\lambda^0 +\BF{f}_\lambda^{s+1}+f_\circ\right)+(\BF{f}_\lambda^1+\cdots +\BF{f}_\lambda^s)\\
&=(\BF{f}_\mu^0+\BF{f}_\mu^{s+1})+(\BF{f}_\mu^1+\cdots +\BF{f}_\mu^s)
=f_\mu.
\end{align*}
This reduces to the proof to the case $s=0$ of ``height two" partitions
$$\lambda=(a,b),\qquad \mu=(a+1,b-1),$$
for which the element $g$ is constructed by (re-)labelling the boxes of Young tableaux \cite{FFFN, FKN24}: 
\begin{equation*}
\begin{tikzpicture}[every node/.style={draw,regular polygon sides=4,minimum size=1cm,line width=0.05em},scale=0.58, transform shape]
        \node at (1,1)  {$N$};
        \node at (1,0)  {$N\!\!-\!1$};
        \node at (2,1)  {$N\!\!-\!2$};
        \node at (2,0)  {$N\!\!-\!3$};
        \node at (3,1)  {$\cdots$};
        \node at (3,0)  {$\cdots$};
        \node at (4,1) {$\BF{r}$+2};
        \node at (4,0)  {$\BF{r}$+1};
        \node at (5,0)  {$\BF{r}$};
        \node at (6,0) {$\cdots$};
        \node at (7,0) {$1$};        
    \end{tikzpicture},\hspace{2cm}
\begin{tikzpicture}[every node/.style={draw,regular polygon sides=4,minimum size=1cm,line width=0.04em},scale=0.58, transform shape]
        \node at (1,0)  {$N$};
        \node at (2,0)  {$N\!\!-\!1$};
        \node at (2,1)  {$N\!\!-\!2$};
        \node at (3,0)  {$N\!\!-\!3$};
        \node at (3,1)  {$\cdots$};
        \node at (4,0)  {$\cdots$};
        \node at (4,1) {$\BF{r}$+2};
        \node at (5,0)  {$\BF{r}$+1};
        \node at (6,0)  {$\BF{r}$};
        \node at (7,0) {$\cdots$};
        \node at (8,0) {1};    
    \end{tikzpicture}.
\end{equation*}
One can take
\begin{align*}
    g=\diag{\frak{a} I_{\mathbf{r}+1},A_1,\dots,A_{a-1},a\frak{a} I_1},
    \quad \text{with}\quad
    A_i=\left(\begin{array}{cc}
       (a-i)\frak{b}&i\frak{a}  \\ -\frak{b}&\frak{a} 
    \end{array} \right),    
\end{align*}
where $I_p$ denotes the identity matrix of size $p$ and $\frak{a},\frak{b} \in \C^\times$ are constants satisfying $\frak{a}^{a+1}\frak{b}^{b-1}b^{b}=1$, \eg\ $(\frak{a},\frak{b})=(b^{-1/(a+1)},b^{-1})$.
Hence, the condition $(\bigstar)$ holds and the assertion follows as a special case of Theorem \ref{Main thm: PQHR}.

\paragraph{\bf{Case II}} 
Let $\lambda,\mu$ be general partitions satisfying the condition \eqref{eq: useful condition}. 
We take the good gradings associated with the Young tableaux as follows. For the partition $\lambda$, we take the Young tableau so that the first $i$ rows on the bottom are \emph{right}-aligned, \ie\ the right-end box of each row has the same $x$-coordinate.
Then, we \emph{left}-align the rows $i,\dots,j$, and finally, the rows $j,\dots, n$ are \emph{right}-aligned (see the left tableau on \eqref{eq: type A young tableau} corresponding to the partition $\lambda=(6,5,3,3,3,2)$ for an example).
The pyramid is then labelled uniquely so that $i$ sits to the right or below of $j$ whenever $i<j$ holds. 
The Young tableau corresponding to the partition $\mu$ is obtained from the previous one by shifting the boxes on the top of the pyramid, \ie\ the rows $j,\dots,n$, by one to the left and falling the right-end box of the $j$-th row down. We illustrate this process with the partition $\mu=(6,6,3,3,2,2)$ in \eqref{eq: type A young tableau} below, where the Young tableau on the right is obtained from the left one. In this example, $i=2$ and $j=5$.
\begin{align}\label{eq: type A young tableau}
\begin{tikzpicture}[every node/.style={draw,regular polygon sides=4,minimum size=1cm,line width=0.04em},scale=0.58, transform shape]
  \foreach \x in {0,...,5} \node[fill=black!80] at (\x,0) {};
  \foreach \x in {1,...,5} \node at (\x,1) {};
  \foreach \x in {1,...,3} \node at (\x,2) {};
  \foreach \x in {1,...,3} \node at (\x,3) {};
  \foreach \x in {1,...,3} \node[fill=gray!50] at (\x,4) {};
  \foreach \x in {2,3} \node[fill=black!80] at (\x,5) {};
\end{tikzpicture},\hspace{2cm}
\begin{tikzpicture}[every node/.style={draw,regular polygon sides=4,minimum size=1cm,line width=0.04em},scale=0.58, transform shape]
  \foreach \x in {0,...,5} \node[fill=black!80] at (\x,0) {};
  \foreach \x in {0} \node[fill=gray!50] at (\x,1) {};
  \foreach \x in {1,...,5} \node at (\x,1) {};
  \foreach \x in {1,...,3} \node at (\x,2) {};
  \foreach \x in {1,...,3} \node at (\x,3) {};
  \foreach \x in {1,2} \node[fill=gray!50] at (\x,4) {};
  \foreach \x in {1,2} \node[fill=black!80] at (\x,5) {};
\end{tikzpicture},
\end{align}
Note that in the example above, the partitions $\lambda$ and $\mu$ are again not adjacent. For an adjacent partition $\nu$ of $\lambda$, we should have started with a different alignment of rows the Young tableau of $\lambda$ and dropped the box from the row $j=5$ onto the row $i'=3$ as illustrated in the following:
\begin{align}\label{eq: type A young tableau adj}
\begin{tikzpicture}[every node/.style={draw,regular polygon sides=4,minimum size=1cm,line width=0.04em},scale=0.58, transform shape]
  \foreach \x in {0,...,5} \node[fill=black!80] at (\x,0) {};
  \foreach \x in {1,...,5} \node[fill=black!80] at (\x,1) {};
  \foreach \x in {3,...,5} \node at (\x,2) {};
  \foreach \x in {3,...,5} \node at (\x,3) {};
  \foreach \x in {3,...,5} \node[fill=gray!50] at (\x,4) {};
  \foreach \x in {4,5} \node[fill=black!80] at (\x,5) {};
\end{tikzpicture},\hspace{2cm}
\begin{tikzpicture}[every node/.style={draw,regular polygon sides=4,minimum size=1cm,line width=0.04em},scale=0.58, transform shape]
  \foreach \x in {0,...,5} \node[fill=black!80] at (\x,0) {};
  \foreach \x in {2} \node[fill=gray!50] at (\x,2) {};
  \foreach \x in {1,...,5} \node[fill=black!80] at (\x,1) {};
  \foreach \x in {3,...,5} \node at (\x,2) {};
  \foreach \x in {3,...,5} \node at (\x,3) {};
  \foreach \x in {3,4} \node[fill=gray!50] at (\x,4) {};
  \foreach \x in {3,4} \node[fill=black!80] at (\x,5) {};
\end{tikzpicture},
\end{align}

The embeddings of partitions
\begin{equation}\label{eq:embedding}
    \lambda^\circ=(\lambda_i,\dots,\lambda_j)\hookrightarrow \lambda,\qquad \mu^\circ=(\mu_i,\dots,\mu_j)\hookrightarrow \mu,
\end{equation}
straightforwardly induces the embedding of the corresponding Young tableaux. 
For example, Young tableaux
in \eqref{eq: sliding the boxes} and \eqref{eq: sliding the boxes adj} are obtained from those in \eqref{eq: type A young tableau} and \eqref{eq: type A young tableau adj} respectively by removing the dark boxes. 
By setting $M=\lambda_i+\cdots+\lambda_j$,  \eqref{eq:embedding} induce the embeddings
$$\tau_{\lambda^\circ}\colon \C^M \rightarrow \C^N,\qquad \tau_{\mu^\circ}\colon \C^M\hookrightarrow \C^N$$
as $\sll_2$-modules, extending the $\ad(f_\lambda)$-actions given by \eqref{eq: constracting the sl2-triple}. 
Abusing notation, let also denote by $\tau_\bullet\colon \End(\C^M) \hookrightarrow \End(\C^N)$ the corresponding embeddings by extending trivially.
By {\bf Case I}, we have the existence of $f_\circ=\widetilde{f}_{\mu^\circ}-f_{\lambda^\circ}$ and $g$ such that $g^{-1}\ast\widetilde{f}_{\mu^\circ}=f_{\mu^\circ}$.
Setting $\widetilde{f}_{\mu}=f_\lambda+\tau_{\lambda^\circ}(f_\circ)$, we find that $(\widetilde{f}_{\mu},\Gamma_\mu)$ is a good pair since $\tau_{\mu^\circ}(g^{-1}) \ast \widetilde{f}_{\mu}=f_\mu$ holds. Hence, the condition $(\bigstar)$ holds, and the assertion follows as a special case of Theorem \ref{Main thm: PQHR}.
\end{proof}
\printbibliography

@article{AF19,
  author = {Arakawa, T. and Frenkel, E.},
  date = {2019},
  journaltitle = {Compos. Math.},
  volume = {155},
  pages = {2235--2262},
  title = {Quantum {{Langlands}} Duality of Representations of {$\mathcal{W}$}-Algebras}
}

@article{Ad19,
  title = {Realizations of Simple Affine Vertex Algebras and Their Modules: The Cases ${\widehat{sl(2)}}$ and ${\widehat{osp(1,2)}}$},
  author = {Adamovi\'c, D.},
  date = {2019},
  journaltitle = {Commun. Math. Phys.},
  volume = {366},
  pages = {1025-–1067}
}

@article{AKM15,
  author = {Arakawa, T. and Kuwabara, T. and Malikov, F.},
  date = {2015},
  journaltitle = {Commun. Math. Phys.},
  volume = {335},
  pages = {143--182},
  title = {Localization of Affine W-Algebras}
}

@article {ACL19,
    AUTHOR = {Arakawa, T. and Creutzig, T. and Linshaw, A. R.},
     TITLE = {{$W$}-algebras as coset vertex algebras},
   JOURNAL = {Invent. Math.},
  FJOURNAL = {Inventiones Mathematicae},
    VOLUME = {218},
      YEAR = {2019},
    NUMBER = {1},
     PAGES = {145--195}
}

@article{AGT10,
  author = {Alday, L. and Gaiotto, D. and Tachikawa, Y.},
  date = {2010},
  journaltitle = {Lett. Math. Phys.},
  volume = {91},
  pages = {167--197},
  title = {Liouville Correlation Functions from Four-Dimensional Gauge Theories}
}

@article{Ara07,
  author = {Arakawa, T.},
  date = {2007},
  journaltitle = {Invent. math.},
  volume = {169},
  pages = {219--320},
  title = {Representation Theory of {$\mathcal{W}$}-Algebras}
}

@article{Ara15a,
  author = {Arakawa, T.},
  date = {2015},
  journaltitle = {Int. Math. Res. Not.},
  volume = {22},
  pages = {11605--11666},
  title = {Associated Varieties of Modules over {{Kac-Moody}} Algebras and {$C_2$}-{{Cofiniteness}} of {$\mathcal{W}$}-{{Algebras}}}
}

@misc{Ara18,
  author = {Arakawa, T.},
  date = {2018},
  eprint = {1811.01577},
  eprinttype = {arxiv},
  primaryclass = {math.RT},
  archiveprefix = {arXiv},
  title = {Chiral Algebras of Class {$\mathcal{S}$}and {{Moore-Tachikawa}} Symplectic Varieties}
}

@misc{BN25,
      title={Inverse Hamiltonian reduction for affine W-algebras in type $A$}, 
      author={Butson, D. and Nair, S.},
      year={2025},
      eprint={2508.18248},
      archivePrefix={arXiv},
      primaryClass={math.RT}
}

@article{BG05,
  title = {Good Grading Polytopes},
  author = {Brundan, J. and Goodwin, S.},
  date = {2005},
  journaltitle = {Proc. Lond. Math. Soc.},
  volume = {94},
  pages = {155--180}
}

@article{BLL15,
  title = {Infinite Chiral Symmetry in Four Dimensions},
  author = {Beem, C. and Lemos, M. and Liendo, P. and Peelaers, W. and Rastelli, L. and family=Rees, given=B. C., prefix=van, useprefix=false},
  date = {2015},
  journaltitle = {Comm. Math. Phys.},
  volume = {336},
  pages = {1359--1433}
}

@article{BPR15,
  author = {Beem, C. and Peelaers, W. and Rastelli, L. and family=Rees, given=B. C., prefix=van, useprefix=false},
  date = {2015},
  journaltitle = {J. High Energ. Phys.},
  volume = {20},
  number = {5},
  title = {Chiral Algebras of Class {$\mathcal{S}$}}
}

@article{BFN16,
  author = {Braverman A. and Finkelberg, M. and Nakajima, H.},
  date = {2016},
  journaltitle = {Astérisque},
  volume = {385},
  title = {Instanton moduli spaces and {W}-algebras}
}

@article{Bry,
  title = {The lattice of integer partitions},
  author = {Brylawski, T.},
  date = {1973},
  journaltitle = {Discrete Math.},
  volume = {6},
  number={3},
  pages = {201--219}
}

@article{CFLN,
  title = {On the structure of $\mathcal{W}$-algebras in type $A$},
  author = {Creutzig, T. and Fasquel, J. and Linshaw, A. R. and Nakatsuka, S.},
JOURNAL = {Jpn. J. Math.},
    VOLUME = {20},
      YEAR = {2025},
    NUMBER = {1},
     PAGES = {1--111},
}

@article{DS85,
  title = {Lie Algebras and Equations of Korteweg-de Vries Type},
  author = {Drinfel’d, V. and Sokolov, V.},
  date = {1985},
  journaltitle = {J. Sov. Math.},
  pages={1975--2036},
  volume={30}
}

@article{DSJKV,
  title = {Integrability of classical affine W-algebras},
  author = {De Sole, A. and Jibladze, M. and Kac, V. and Valeri, D.},
  date = {2021},
  journaltitle = {Transform. Group.},
  pages={479--500},
  volume={26}
}

@article{DK06,
  author = {De Sole, A. and Kac, V.},
  date = {2006},
  journaltitle = {Japan. J. Math.},
  volume = {1},
  pages = {137--261},
  title = {Finite vs Affine {$\mathcal{W}$}-Algebras}
}

@incollection{EK05,
  title = {Classification of Good Gradings of Simple {{Lie}} Algebras},
  booktitle = {Lie Groups and Invariant Theory},
  author = {Elashvili, A. and Kac, V.},
  date = {2005},
  series = {Amer. {{Math}}. {{Soc}}. {{Transl}}. {{Ser}}. 2},
  edition = {Ernest Vinberg Editor},
  volume = {213},
  pages = {85--104},
  publisher = {{American Mathematical Society}}
}

@misc{FKN24,
      title={On Virasoro-type reductions and inverse Hamiltonian reductions for $W$-algebras and $W_\infty$-algebras},
      author={Fasquel, J. and Kovalchuk, V. and Nakatsuka, S.},
      year={2024},
      eprint={2411.10694},
      archivePrefix={arXiv},
      primaryClass={math.QA}
}

@misc{FFFN,
      title={Connecting affine $\W$-algebras: A case study on $\sll_4$},
      author={Fasquel, J. and Fehily, Z. and Fursman, E. and Nakatsuka, S.},
      year={2024},
      eprint={2408.13785 },
      archivePrefix={arXiv},
      primaryClass={math.QA}
}

@misc{FRR,
      title={Modularity of admissible-level $\mathfrak{sl}_3$ minimal models with denominator $2$},
      author={Fasquel, J. and Raymond, C. and Ridout, D.},
      year={2024},
      eprint={2406.10646},
      archivePrefix={arXiv},
      primaryClass={math.QA}
}

@article{FF90,
  title = {Quantization of the {{Drinfeld-Sokolov}} Reduction},
  author = {Feigin, B. and Frenkel, E.},
  date = {1990},
  journaltitle = {Phys. Lett. B},
  volume = {246},
  pages = {75--81}
}

@book{Fre07,
  title = {Langlands Correspondence for Loop Groups},
  author = {Frenkel, E.},
  date = {2007},
  series = {Cambridge {{Studies}} in {{Advanced Mathematics}}},
  volume = {103},
  publisher = {{Cambridge University Press}},
  location = {{Cambridge}}
}

@misc{Gai16,
  title = {Quantum {{Langlands}} Correspondence},
  author = {Gaitsgory, D.},
  date = {2016},
  eprint = {1601.05279},
  eprinttype = {arxiv},
  primaryclass = {math.AG},
  archiveprefix = {arXiv}
}

@article{Gen20,
title = {Screening operators and parabolic inductions for affine W-algebras},
journal = {Adv. Math.},
volume = {369},
pages = {no. 107179},
date = {2020},
author = {Genra, N.}
}

@article{GGS17,
title = {Generalized and degenerate Whittaker models},
journal = {Compos. Math.},
volume = {153(2)},
pages = {223–256},
date = {2017},
author = {Gomez, R. and Gourevitch, D. and Sahi, S.}
}

@article{GJ23,
      title={Reduction by stages for finite W-algebras}, 
      author={Genra, N. and Juillard, T.},
      journal={Math. Z.},
      year={2024},
      volume={308},
      number={15}
}

@misc{GJ25,
      title={Reduction by stages for affine W-algebras}, 
      author={Genra, N. and Juillard, T.},
      date={2025},
      eprint = {2501.04501},
      eprinttype = {arxiv},
      primaryclass = {math.RT},
      archiveprefix = {arXiv}
}

@article{KRW03,
  title = {Quantum Reduction for Affine Superalgebras},
  author = {Kac, V. and Roan, S. and Wakimoto, M.},
  date = {2003},
  journaltitle = {Comm. Math. Phys.},
  volume = {241},
  pages = {307--342}
}

@article{KW04,
  title = {Quantum Reduction and Representation Theory of Superconformal Algebras},
  author = {Kac, V. and Wakimoto, M.},
  date = {2004},
  journaltitle = {Adv. Math.},
  volume = {185},
  pages = {400--458}
}

@article {Kostant78,
    AUTHOR = {Kostant, Bertram},
     TITLE = {On {W}hittaker vectors and representation theory},
   JOURNAL = {Invent. Math.},
  FJOURNAL = {Inventiones Mathematicae},
    VOLUME = {48},
      YEAR = {1978},
    NUMBER = {2},
     PAGES = {101--184},
}

@article {Losev10,
    AUTHOR = {Losev, I.},
     TITLE = {Quantized symplectic actions and {$W$}-algebras},
   JOURNAL = {J. Amer. Math. Soc.},
  FJOURNAL = {Journal of the American Mathematical Society},
    VOLUME = {23},
      YEAR = {2010},
    NUMBER = {1},
     PAGES = {35--59},
}

@book{Morgan14,
    AUTHOR = {Morgan, S.},
     TITLE = {Quantum {H}amiltonian reduction of {W}-algebras and category
              {O}},
      NOTE = {Thesis (Ph.D.)--University of Toronto (Canada)},
 PUBLISHER = {ProQuest LLC, Ann Arbor, MI},
      YEAR = {2014},
      ISBN = {978-1339-37342-3},
    PAGETOTAL = {66}
}

@article{Pre02,
  title = {Special Transverse Slices and Their Enveloping Algebras},
  author = {Premet, A.},
  date = {2002},
  journaltitle = {Adv. Math.},
  volume = {170},
  pages = {1--55}
}

@article{SV13,
  author = {Schiffmann, O. and Vasserot, E.},
  date = {2013},
  journaltitle = {Publ. Math. IHÉS},
  volume = {118},
  pages={213--342},
  title = {Cherednik algebras, {W}-algebras and the equivariant cohomology of the moduli space of instantons on ${A^2}$}
}

@article{SXY17,
  author = {Song, J. and Xie, D. and Yan, W.},
  date = {2017},
  journaltitle = {J. High Energ. Phys.},
  volume = {12},
  title = {Vertex Operator Algebras of {{Argyres-Douglas}} Theories from {$M5$}-{{Branes}}}
}

\end{document}